\documentclass[a4paper, 12pt]{article}

\usepackage[sort&compress]{natbib}
\bibpunct{(}{)}{;}{a}{}{,} 

\usepackage{amsthm, amsmath, amssymb, mathrsfs, multirow, url, subfigure}
\usepackage{graphicx} 
\usepackage{ifthen} 
\usepackage{amsfonts}
\usepackage[usenames]{color}
\usepackage{fullpage}
\usepackage{makecell}
\usepackage{array, booktabs, caption}

\theoremstyle{plain} 
\newtheorem{thm}{Theorem}
\newtheorem{cor}{Corollary}
\newtheorem{prop}{Proposition}

\theoremstyle{definition}

\theoremstyle{remark}

\newcommand{\RR}{\mathbb{R}}

\newcommand{\E}{\mathsf{E}}

\newcommand{\prob}{\mathsf{P}}

\newcommand{\eps}{\varepsilon}
\renewcommand{\phi}{\varphi}

\newcommand{\rank}{\mathrm{rank}}

\newcommand{\nm}{\mathsf{N}}

\newcommand{\chisq}{\mathsf{ChiSq}}

\newcommand{\lkappa}{\underline{\kappa}}
\newcommand{\ukappa}{\overline{\kappa}}

\renewcommand{\S}{\mathcal{S}}

\title{Empirical priors for prediction in sparse high-dimensional linear regression\footnote{Running title: {\em Empirical priors for prediction}}}
\author{
Ryan Martin\footnote{Department of Statistics, North Carolina State University, 2311 Stinson Dr., Raleigh, NC 27695. Email: {\tt rgmarti3@ncsu.edu}, {\tt ytang22@ncsu.edu}} \; and \; Yiqi Tang$^\dagger$
}
\date{\today}

\begin{document}

\maketitle 

%\doublspacing

\begin{abstract}
%Often the primary goal of fitting a regression model is prediction, but the majority of work in recent years focuses on inference tasks, such as estimation and feature selection.  
In this paper we adopt the familiar sparse, high-dimensional linear regression model and focus on the important but often overlooked task of prediction.  In particular, we consider a new empirical Bayes framework that incorporates data in the prior in two ways: one is to center the prior for the non-zero regression coefficients and the other is to provide some additional regularization.  We show that, in certain settings, the asymptotic concentration of the proposed empirical Bayes posterior predictive distribution is very fast, and we establish a Bernstein--von Mises theorem which ensures that the derived empirical Bayes prediction intervals achieve the targeted frequentist coverage probability.  The empirical prior has a convenient conjugate form, so posterior computations are relatively simple and fast.  Finally, our numerical results demonstrate the proposed method's strong finite-sample performance in terms of prediction accuracy, uncertainty quantification, and computation time compared to existing Bayesian methods.  

\smallskip

\end{abstract}

\section{Introduction}
\label{S:intro}

Consider a linear regression model 
\begin{equation}
\label{eq:reg}
y = X \beta + \sigma z, 
\end{equation}
where $y$ is a $n \times 1$ vector of response variables, $X$ is a $n \times p$ matrix of explanatory variables, $\beta$ is a $p \times 1$ vector of regression parameters, $\sigma > 0$ is an unknown scale parameter, and $z$ is a $n \times 1$ vector of independent standard normal errors.  Here, our interest is in the {high-dimensional} setting where $p \gg n$, and our particular aim is to predict the value of a new response $\tilde y \in \RR^d$ at a given $\tilde X \in \RR^{d \times p}$, $d \geq 1$, an important and challenging problem in these high-dimensional scenarios.  

An initial obstacle to achieving this aim is that the model above cannot be fit without some additional structure.  As is common in the literature, we will assume a {sparsity} structure on the high-dimensional $\beta$ vector.  That is, we will assume that most of the entries in $\beta$ are zero; this will be made more precise in the following sections.  With this assumed structure, a plethora of methods are now available for estimating a sparse $\beta$, e.g., lasso \citep{tibshirani1996}, adaptive lasso \citep{zou2006}, SCAD \citep{fanli2001}, and others; moreover, software is available to carry out the relevant computations easily and efficiently.  Given an estimator of $\beta$, it is conceptually straightforward to produce a point prediction of a new response.  However, the regularization techniques employed by these methods cause the estimators to have non-regular distribution theory \citep[e.g.,][]{potscher.leeb.2009}, so results on uncertainty quantification, i.e., coverage properties of prediction intervals, are few in number; but see \citet{leeb2006, leeb2009} and the references therein.  

On the Bayesian side, given a full probability model, it is conceptually straightforward to obtain a {predictive distribution} for the new response and suggest some form of uncertainty quantification, but there are still a number of challenges.  First, in high-dimensional cases such as this, the choice of prior matters, so specifying prior distributions that lead to desirable operating characteristics of the posterior distribution, e.g., optimal posterior concentration rates, is non-trivial.  \citet{castillo.schmidt.vaart.reg} and others have demonstrated that in order to achieve the optimal concentration rates, the prior for the non-zero $\beta$ coefficients must have sufficiently heavy tails, in particular, heavier than the conjugate Gaussian tails.  This constraint leads to the second challenge, namely, computation of the posterior distribution.  While general Markov chain Monte Carlo (MCMC) methods are available, the individual steps can be expensive and the chain can be slow to converge.  Some believe these computations to be prohibitively slow for priors that include a discrete component for the zero coefficients, so they prefer continuous shrinkage priors like the horseshoe \citep{carvalho.polson.scott.2010} and Dirichlet--Laplace \citep{dunson.shrinkage}.  In any case, even if the first two challenges can be overcome and a predictive distribution for $\tilde y$ can be obtained, it is not automatic that the prediction intervals from this distribution provide valid uncertainty quantification, i.e., that the posterior 95\% predictive interval will contain the to-be-observed value of $\tilde y$ with probability 0.95.  

The computational difficulties mentioned above stem from the need to work with the heavy-tailed priors that yield desired posterior concentration properties.  Inspired by the insight that prior tails would be irrelevant if the prior center was appropriately chosen, \citet{martin.mess.walker.eb} developed a new empirical Bayes approach for high-dimensional regression that incorporates data to both center the prior and to provide some extra regularization.  Their approach is powerful because it allows for conjugate priors to be used, which drastically speeds up computation, but without sacrificing on the desirable concentration rate properties enjoyed by the fully Bayesian approach with heavy-tailed priors.  See, also, \citet{martin.walker.deb}.  

Our goal in the present paper is to investigate the performance of the empirical Bayes approach in \citet{martin.mess.walker.eb} in the context of predicting a new response.  First, we review their empirical Bayes formulation in Section~\ref{S:review}, adding some details about the unknown-$\sigma$ case.  Next, in Section~\ref{S:predictive}, we turn to the prediction task and show that, thanks to the empirical prior's conjugacy, the corresponding empirical Bayes predictive distribution has a relatively simple form and can be easily and efficiently sampled via standard Monte Carlo techniques.  Theoretical properties are investigated in Section~\ref{SS:theory} and, in particular, we show that our empirical Bayes predictive distribution has fast convergence rates, nearly parametric in some cases, both for in- and out-of-sample prediction settings.  Moreover, under reasonable conditions, we establish a Bernstein--von Mises theorem, which implies that the derived posterior prediction intervals have the target coverage probability asymptotically.  In Section~\ref{S:examples} we demonstrate, in both real- and simulated-data examples, that the proposed empirical Bayes framework provides accurate point prediction, valid prediction uncertainty quantification, and fast computation across various settings compared to a number of existing methods.  Finally, some concluding remarks are given in Section~\ref{S:discuss}.  We have also  developed an R package, {\tt ebreg} \citep{ebreg.R}, that provides users with tools for estimation and variable selection, as described in \citet{martin.mess.walker.eb}, and the tools for prediction as presented here.

\section{Empirical prior for regression}
\label{S:review}

\subsection{Known $\sigma^2$}
\label{SS:known}

Here we review the empirical prior approach for sparse, high-dimensional regression laid out in \citet{martin.mess.walker.eb}.  Like \citet{castillo.schmidt.vaart.reg} and others, they focus on the known-$\sigma^2$ case, so we present their formulation here.  Adjustments to handle the more realistic unknown-$\sigma^2$ case are described in Section~\ref{SS:unknown}.  

Under the sparsity assumption, it is natural to decompose the high-dimensional vector $\beta$ as $(S, \beta_S)$, where $S \subseteq \{1,2,\ldots,p\}$ is the {configuration} of $\beta$, i.e., the set of indices corresponding to non-zero/active coefficients, and $\beta_S$ is the $|S|$-vector of non-zero values; here $|S|$ denotes the cardinality of the finite set $S$.  This decomposition suggests a hierarchical model with a marginal prior for $S$ and a conditional prior for $\beta_S$, given $S$. 

For the marginal prior for $S$, we take the mass function 
\begin{equation}
\label{eq:S.prior}
\pi(S) = \textstyle \binom{p}{|S|}^{-1} q_n(|S|), \quad S \subset \{1,2,\ldots,p\}, \quad |S| \leq R, 
\end{equation}
where $q_n$ is a mass function on $\{0,1,\ldots,R\}$, which we take to be 
\begin{equation}
\label{eq:complexity}
q_n(s) \propto (cp^a)^{-s}, \quad s=0,1,\ldots,R,
\end{equation}
with $R=\rank(X)$ and $(a,c)$ some hyperparameters to be specified; see Section~\ref{S:examples}.  This corresponds to a truncated geometric prior on the configuration size and a uniform prior on all configurations of the given size; see, also, \citet{castillo.schmidt.vaart.reg}.  

There is an assumption hidden in the definition \eqref{eq:S.prior} that deserves comment.  The prior does not support all possible configurations, only those of size no more than $R \leq n \ll p$.  The rationale for this restriction is that the true value of $\beta$, namely $\beta^\star$, is assumed to be sparse, i.e., the true configuration, $S^\star = S_{\beta^\star}$, is of size much smaller than $n$, so there is no reason not to incorporate an assumption like ``$|S^\star| \leq n$'' into the prior.  Since $R \leq n$, the assumption ``$|S^\star| \leq R$'' encoded in \eqref{eq:S.prior} is generally stronger.  However, $R=n$ is typical, e.g., if the rows of $X$ are iid $p$-vectors with positive definite covariance matrix, in which case ``$|S^\star| \leq R$'' and ``$|S^\star| \leq n$'' are equivalent.  Moreover, nothing significant changes about the theory if $|S^\star| \leq R < n$; see, e.g., \citet[][Theorem~3]{abramovich.grinshtein.2010}.  But if it happens that $R < |S^\star| < n$, then the compatibility condition described in \citet{castillo.schmidt.vaart.reg} fails, so even an oracle who knows which variables are important would be unable to give a fully satisfactory solution to the problem.  Since one inevitably must assume that $|S^\star| \leq R$ to establish any good properties, we opt to include such an assumption in the prior construction, via \eqref{eq:S.prior} and \eqref{eq:complexity}.  

The empirical or data-dependent element comes in the conditional prior for $\beta_S$, given $S$.  That is, set 
\begin{equation}
\label{eq:beta.prior.S}
\beta_S \mid S \sim \nm_{|S|}\bigl( \hat\beta_S, \sigma^2 \gamma^{-1} (X_S^\top X_S)^{-1}\bigr), 
\end{equation}
where $X_S$ is the $n \times |S|$ submatrix of $X$ with only the columns corresponding to the configuration $S$, $\hat\beta_S$ is the least squares estimate based on design matrix $X_S$, and $\gamma > 0$ is a precision parameter to be specified.  Except for being centered on the least squares estimator, this closely resembles Zellner's $g$-prior \citep[e.g.,][]{zellner1986}.  Again, the idea behind a data-dependent centering is to remove the influence of the prior tails on the posterior concentration, which requires use of the data.  See \citet{martin.walker.eb, martin.walker.deb} and \citet{martin.mess.walker.eb} for more on this point. 

Putting the two pieces together, we have the following empirical prior for $\beta$: 
\begin{equation}
\label{eq:beta.prior}
\Pi_n(d\beta) = \sum_S \pi(S) \nm_{|S|}\bigl( d\beta_S \mid \hat\beta_S, \sigma^2 \gamma^{-1} (X_S^\top X_S)^{-1} \bigr) \otimes \delta_{0_{S^c}}(d\beta_{S^c}), 
\end{equation}
where $\delta_{0_{S^c}}$ denotes a Dirac point-mass distribution at the origin in the $|S^c|$-dimensional space, and $\mu \otimes \nu$ is the product of two measures $\mu$ and $\nu$.  This is a {spike-and-slab prior} where the spikes are point masses and the slabs are conjugate normal densities, which have nice computational properties but are centered at a convenient estimator to eliminate the thin-tail effect on the posterior concentration rate.  

Next we combine this prior with the likelihood in almost the usual way.  That is, for a constant $\alpha \in (0,1)$, define the corresponding empirical Bayes posterior $\Pi^n$ for $\beta$ as 
\begin{equation}
\label{eq:update}
\Pi^n(d\beta) \propto L_n(\beta)^\alpha \, \Pi_n(d\beta), 
\end{equation}
where 
\begin{equation}
\label{eq:likelihood}
L_n(\beta) = \nm_n(y \mid X\beta, \sigma^2 I) \propto \exp\{-\tfrac{1}{2\sigma^2} \|y - X\beta\|^2\}, 
\end{equation}
is the likelihood, with $\|\cdot\|$ the $\ell_2$-norm on $\RR^n$.  The power $\alpha$ is unusual, but the role it plays is to flatten out the posterior, effectively discounting the data slightly.  \citet{martin.mess.walker.eb} describe this as a regularization that prevents the posterior from chasing the data too closely, and similar discounting ideas have been used for robustness purposes in certain misspecified models \citep[e.g.,][]{grunwald.ommen.scaling, syring.martin.scaling, holmes.walker.scaling}.  In our present context, the $\alpha$ discount is a technical device to help the posterior adapt to the unknown sparsity \citep[see][]{martin.walker.deb}.
%, and there are some potential benefits to this discounting when it comes to uncertainty quantification; see \citet{ebcvg} and Section~\ref{SS:theory} below.  
For those readers uncomfortable with the power likelihood in \eqref{eq:update}, an equivalent representation is as a genuine Bayesian update 
\[ \Pi^n(d\beta) \propto L_n(\beta) \, \Pi_n^{\text{reg}}(d\beta), \]
where 
\[ \Pi_n^{\text{reg}}(d\beta) \propto L_n(\beta)^{-(1-\alpha)} \, \Pi_n(d\beta) \]
is a version of the above empirical prior with an extra data-driven regularization, penalizing those $\beta$ values that fit the data too well when $L_n(\beta)$ is large. In any case, we recommend taking $\alpha \approx 1$ in applications so there is no practical difference between our proposal and a closer-to-genuine Bayes posterior with $\alpha=1$.  
%In particular, we take $\alpha=0.99$ in all of our numerical examples.  
The end result is a posterior distribution, $\Pi^n$, for $\beta$ that depends on $\alpha$, $\gamma$, and, in this case, the known $\sigma^2$.  

In terms of basic first-order properties of $\Pi^n$, such as asymptotic concentration rates, the results are not sensitive to the choice of $\alpha$ and $\gamma$.  That is, the same concentration rates are obtained for all $\alpha \in (0,1)$ and all $\gamma > 0$.  However, for higher-order properties, such as coverage of posterior credible sets, some conditions on $\alpha$ and $\gamma$ are required.  As discussed in Section~\ref{SS:theory}, we will require $\alpha + \gamma \leq 1$, so if $\alpha$ is close to 1, then $\gamma$ must be close to 0.  The apparent impact of $\gamma \approx 0$ is that the conditional prior for $\beta_S$, given $S$, in \eqref{eq:beta.prior.S} is wide and ``non-informative'' in a traditional sense.  But, for example, if we treat the rows of $X$ as iid samples with mean 0 and covariance matrix $\Sigma$, then for any fixed $S$, we have that $(X_S^\top X_S)^{-1} = n^{-1} \widehat\Sigma_S^{-1}$, where $\Sigma_S$ is the submatrix corresponding to configuration $S$ and the ``hat'' indicates an estimate based on the sample in $X$.  So the prior for $\beta_S$, given $S$, has a data-driven center and variance roughly $O(n^{-1})$.  Therefore, choosing $\gamma \approx 0$ as suggested by the theory also has some practical intuition: it effectively prevents the conditional prior for $\beta_S$, given $S$, from being ``too informative.''  

An important practical advantage of this formulation is that $\Pi^n$ is relatively simple.  Indeed, the conditional posterior distribution for $\beta_S$, given $S$, is just 
\begin{equation}
\label{eq:beta.post}
\pi^n(\beta_S \mid S) = \nm_{|S|}\bigl(\beta_S \mid \hat\beta_S, \tfrac{\sigma^2}{\alpha + \gamma} (X_S^\top X_S)^{-1} \bigr). 
\end{equation}
For variable selection-related tasks, the marginal posterior for the configuration, $S$, is the relevant object, and a closed-form expression is available:
%\begin{equation}
%\label{eq:S.post}
\[ \pi^n(S) \propto \pi(S) \bigl(\tfrac{\gamma}{\alpha + \gamma} \bigr)^{|S|/2} \exp\{-\tfrac{\alpha}{2\sigma^2} \|y - \hat y_S\|^2\}, \]
%\end{equation}
where $\hat y_S$ is the fitted response based on the least squares fit to $(y, X_S)$. \citet{martin.mess.walker.eb} propose a Metropolis--Hastings procedure to sample from $S$ and, if $\beta$ samples are also desired, then one can augment the $S$ sampler by sampling from the conditional posterior for $\beta_S$, given $S$, along the way.

\subsection{Unknown $\sigma^2$}
\label{SS:unknown}

For the realistic case where the error variance is unknown, there are different strategies one can employ.  The simplest strategy, taken in \citet{martin.mess.walker.eb}, is to construct an estimator, $\hat\sigma^2$, and plug it in to the known-$\sigma^2$ formulas above.  They used a lasso-driven estimator, discussed in \citet{reid.tibshirani.friedman.2014}, in their numerical examples, and their method had very good performance.  But variance estimates post-selection can be unreliable  \citep[e.g.,][]{hkm2018.overfit}, which can impact other posterior summaries, such as credible regions, so we want to consider an alternative based on a prior distribution for $\sigma^2$.  

Consider an inverse gamma prior for $\sigma^2$, with density 
\[ \pi(\sigma^2) = b_0^{a_0} \Gamma(a_0)^{-1} (\sigma^2)^{-(a_0 + 1)} e^{-b_0/ \sigma^2}, \quad \sigma^2 > 0, \]
where $a_0$ and $b_0$ are fixed shape and scale parameters, respectively.  Incorporating this into the prior formulation described above, expanding the likelihood in \eqref{eq:likelihood} as  
\[ L_n(\beta, \sigma^2) = \nm_n(y \mid X\beta, \sigma^2 I) \propto (\sigma^2)^{-1/2} \exp\{-\tfrac{1}{2\sigma^2} \|y-X\beta\|^2\}, \]
to include $\sigma^2$, and combining the two as in \eqref{eq:update}, the following properties of the posterior distribution are easy to verify. First, the conditional posterior for $\beta_S$, given $S$ and $\sigma^2$ is exactly as in \eqref{eq:beta.post}; second, the conditional posterior distribution for $\sigma^2$, given $S$, is again inverse gamma with $\text{shape} = a_0 + \tfrac{\alpha n}{2}$ and $\text{scale} = b_0 + \tfrac{\alpha}{2} \|y - \hat y_S\|^2$; and, finally, the marginal posterior for $S$ is 
\begin{equation}
\label{eq:S.post}
\pi^n(S) \propto \pi(S) \bigl( \tfrac{\gamma}{\alpha + \gamma} \bigr)^{|S|/2} \bigl\{ b_0 + \tfrac{\alpha}{2} \|y-\hat y_S\|^2 \bigr\}^{-(a_0 + \alpha n / 2)}. 
\end{equation}
Therefore, the MCMC strategy described above to evaluate the posterior can proceed immediately with this alternative expression for $\pi^n$.

\section{Empirical Bayes predictive distribution}
\label{S:predictive}

%\subsection{Definition and computation}
%\label{SS:predictive}

Given the empirical Bayes posterior defined above, either for known or unknown $\sigma^2$, we can immediately obtain a corresponding predictive distribution.  Consider a pair $(\tilde X, \tilde y)$ where $\tilde X \in \RR^{d \times p}$ is a given matrix of explanatory variable values at which we seek to predict the corresponding response $\tilde y \in \RR^d$.  

If $\sigma^2$ were known, or if a plug-in estimator is used, then the conditional posterior predictive distribution of $\tilde y$, given $S$, is familiar, and given by 
\[ f_{\tilde X}^n(\tilde y \mid S) = \nm_d\bigl( \tilde y \mid \tilde X_S \hat\beta_S, \, \sigma^2 I_d + \tfrac{\sigma^2}{\alpha + \gamma} \tilde X_S (X_S^\top X_S)^{-1} \tilde X_S^\top \bigr). \]
%This is a Rao--Blackwellized estimator of the predictive distribution since we have analytically integrated out $\beta_S$ for a given $S$.  
To obtain the predictive distribution for $\tilde y$, we simply need to integrate out $S$ with respect to its posterior, i.e., 
\begin{equation}
\label{eq:pred}
f_{\tilde X}^n(\tilde y) = \sum_S \pi^n(S) \, f_{\tilde X}^n(\tilde y \mid S). 
\end{equation}
Of course, one cannot evaluate this sum because there are too many terms, but it is possible to use samples of $S$ from the marginal posterior, $\pi^n(\cdot)$, to get a Monte Carlo approximation of the predictive density $f_{\tilde X}^n(\tilde y)$ at some set values of $\tilde y$.  Alternatively, one can augment the aforementioned MCMC algorithm to sample $\tilde y$ from $f_{\tilde X}^n(\tilde y \mid S)$ along the Markov chain; see below.  Having a sample from the predictive distribution is advantageous when it comes to creating posterior credible sets for prediction.  For example, in the $d=1$ case, a 95\% posterior prediction interval can be found by computing quantiles of the sample taken from the predictive distribution.  

Very little changes when the inverse gamma prior for $\sigma^2$ is adopted. Indeed, the predictive density $f_{\tilde X}^n(\tilde y\mid S)$ is just the density for a $d$-variate Student-t distribution, with $2a_0 + \alpha n$ degrees of freedom, location $\tilde X_S \hat\beta_S$, and scale matrix 
\begin{equation}
\label{eq:scale}
\tfrac{b_0 + (\alpha/2)\|y - \hat y_S\|^2}{a_0 + \alpha n / 2} \bigl( I_d + \tfrac{1}{\alpha + \gamma} \tilde X_S (X_S^\top X_S)^{-1} \tilde X_S^\top \bigr). 
\end{equation}
From here, sampling from the predictive \eqref{eq:pred} can proceed exactly like before, with straightforward modifications to accommodate the Student-t instead of normal shape.  

For computation, \citet{martin.mess.walker.eb} recommend a simple Metropolis--Hastings scheme based on the marginal posterior distribution for $S$.  If the focus was strictly on variable selection tasks, so that only the posterior distribution of $S$ were relevant, then a shotgun stochastic search strategy could also be taken, as in \citet{ecap}, which avoids sampling on the complex $S$-space.  But here our focus is on prediction, so we want the samples of $S$ so that we can readily sample from the conditional predictive distribution of $\tilde y$, given $S$, along the way.  Specifically, if $q(S' \mid S)$ is a proposal function, then a single iteration of our Metropolis--Hastings sampler goes as follows:
\begin{enumerate}
\item Given a current state $S'$, sample $S_\text{tmp} \sim q(\cdot \mid S')$.
\vspace{-2mm}
\item Set $S=S_\text{tmp}$ with probability 
\[ \min\Bigl\{1, \frac{\pi^n(S')}{\pi^n(S_{\text{tmp}})} \frac{q(S_{\text{tmp}} \mid S')}{q(S' \mid S_{\text{tmp}})} \Bigr\}, \]
where $\pi^n$ is as in \eqref{eq:S.post}; otherwise, set $S=S'$.
\vspace{-2mm}
\item Sample $\tilde y$ from a $d$-variate Student-t distribution with $2a_0 + \alpha n$ degrees of freedom, location $\tilde X_S \hat\beta_S$, and scale matrix \eqref{eq:scale} depending on the given $S$.
\end{enumerate}
Repeating this process $T$ times, we obtain an approximate sample $\{\tilde y^{(t)}: t=1,\ldots,T\}$ from the predictive distribution \eqref{eq:pred} corresponding to a given $d \times p$ matrix $\tilde X$ of covariates at which prediction is desired.  In our implementation, we use a symmetric proposal distribution $q(S \mid S')$, i.e., one that samples $S$ uniformly from those models that differ from $S'$ in exactly one position, which simplifies the acceptance probability above since the $q$-ratio is identically 1.  

Of course, the quality of samples from the predictive distribution depends on that of the samples from the marginal posterior distribution for $S$.  It helps that there is a closed-form expression for $\pi^n(\cdot)$, but the configuration space is still very large and complicated, making it virtually impossible for a Markov chain to do a complete exploration in any reasonable amount of time.  Fortunately, a complete exploration of the space is not necessary.  The theory presented in \citet{martin.mess.walker.eb}---see, also, Appendix~\ref{S:bernoulli}---says that $\pi^n(S)$ will tend to be largest at/near the true configuration.  So with a warm start, based on, say, the lasso configuration, the proposed MCMC algorithm quickly explores the subspace of plausible configurations and, in turn, leads to high-quality predictions.  Justification for this claim is based on the strong empirical performance in Section~\ref{S:examples} below, in \citet{martin.mess.walker.eb}, and in other settings, e.g., \citet{lee.lee.lin.deb}, \citet{eb.gwishart}, \citet{ebcvg}, and \citet{ebpiecep}.

\section{Asymptotic properties}
\label{SS:theory}

\subsection{Setup}
\label{SS:setup}

The goal here is to explore the asymptotic properties of the empirical Bayes predictive distribution defined above.  In particular, in Section~\ref{SSS:rate} we bound the rate at which the posterior distribution $\Pi^n$ in \eqref{eq:update} concentrates on $\beta$ vectors that lead to accurate prediction of a new observation which, in certain cases, despite the high dimensionality, is close to the parametric root-$n$ rate.  Also, in Section~\ref{SSS:uq}, we investigate distributional approximations of the predictive distribution and corresponding uncertainty quantification properties.  Some of our results that follow rely on details presented in \citet{martin.mess.walker.eb} so, for the reader's convenience, we summarize the relevant points in Appendix~\ref{S:bernoulli}.  

To fix notation, etc., let $\beta^\star$ denote the true $p$-vector of regression coefficients with configuration $S_{\beta^\star} = \{j: \beta_j^\star \neq 0\}$.  Since $p \gg n$, quality estimation is hopeless without some underlying low-dimensional structure, and here the relevant low-dimensional structure is {\em sparsity}, i.e., the size $|S_{\beta^\star}|$ of the true configuration is small relative to $n$ or, more generally, to the rank $R$.  Rates in such problems are determined by the triple $(n,p,s^\star)$, where $s^\star$ is a sequence controlling the sparsity, increasing slowly with $n$.  \citet{verzelen2012} splits the entire class of problems into two cases, namely, {\em high-} and {\em ultra-high-dimensional}, based on whether $s^\star \log (p/s^\star)$ is small or large relative to $n$.  He shows that, for prediction-related tasks, for true $\beta^\star$'s that are $s^\star$-sparse, the optimal rate $\eps_n$ satisfies 
\[ \eps_n^2 = \min\{n^{-1} s^\star \log(p / s^\star), 1\}. \]
Here the phase transition between high- and ultra-high-dimensional problems is clear, in particular, that there is a limit to how accurate predictions can be in the latter case.  Since we aim for statements like ``the Hellinger distance between the true density and the predictive density in \eqref{eq:pred} is $\lesssim \eps_n$'' (e.g., Corollary~\ref{thm:predictive.rate}), and such conclusions are not meaningful if $\eps_n \not\to 0$, we will focus exclusively on Verzelen's ordinary high-dimensional case.  That is, we make the following {\em standing assumption:}
\begin{equation}
\label{eq:standing}
\text{$(n,p,s^\star)$ satisfies $s^\star \log (p/s^\star) = o(n)$ as $n \to \infty$}. 
\end{equation}
In other settings, like in Theorem~\ref{thm:oos.rate} and its corollary below, where it is necessary to separate $\beta^\star$ from $X\beta^\star$, additional assumptions about $X$ and $S^\star$ are required.  Fortunately, as we discuss in more detail below, it is known that such assumptions are not unreasonable in settings that satisfy \eqref{eq:standing}. 

As is typically done in theoretical analyses of the $p \gg n$ problem \citep[e.g.,][]{castillo.schmidt.vaart.reg}, we work in the case of known error variance.  Moreover, for simplicity and consistency with the existing literature, we assume that the rank $R$ of $X$ equals $n$.  Finally, in what follows, we will write ``$\E_{\beta^\star}$'' to denote expectation with respect to the $n$-vector $y$ defined in \eqref{eq:reg} with true regression coefficient $\beta^\star$.

\subsection{Concentration rates}
\label{SSS:rate}

For ease of presentation, in Theorem~\ref{thm:hellinger.rate} and its corollary, we focus on $d=1$, so the $\tilde X$ matrix can be replaced by a (column) $p$-vector $\tilde x$.  The results hold more generally, but they are rather cumbersome to present, so we defer those details to Appendix~\ref{S:dim}.

For a given $\tilde x \in \RR^p$ and particular values $\beta$ and $\beta^\star$, let $h_{\tilde x}(\beta^\star,\beta)$ denote the Hellinger distance between $\nm(\tilde x^\top \beta, \sigma^2)$ and $\nm(\tilde x^\top \beta^\star, \sigma^2)$.  Following \citet{dunson.compreg.2015}, define an unconditional Hellinger distance
\[ h(\beta^\star, \beta) = \Bigl\{ \int h_{\tilde x}^2(\beta^\star, \beta) \,Q_n(d\tilde x) \Bigr\}^{1/2}, \]
where $Q_n$ is the empirical distribution of those $p$-vectors that fill the rows of $X$.  Then the following theorem establishes the asymptotic concentration rate of the proposed empirical Bayes posterior relative to the prediction-focused metric $h(\beta^\star, \beta)$.  

\begin{thm}
\label{thm:hellinger.rate}
Let $\Pi^n$ be the empirical Bayes posterior defined in \eqref{eq:update}, and let $s^\star$ be a sequence such that \eqref{eq:standing} holds.  Then there exists constants $G,M > 0$ such that 
\[ \sup_{\beta^\star: |S_{\beta^\star}|=s^\star} \E_{\beta^\star} \Pi^n(\{\beta \in \RR^p: h(\beta^\star, \beta) > M\eps_n\}) \lesssim e^{-G n \eps_n^2} \to 0, \]
where $\eps_n^2 = n^{-1} s^\star \log (p/s^\star) \to 0$. 
\end{thm}

\begin{proof}
Without loss of generality, assume $\sigma^2=1$.  Let $k_{\tilde x}(\beta^\star,\beta)$ denote the Kullback--Leibler divergence of $\nm(\tilde x^\top \beta, 1)$ from $\nm(\tilde x^\top \beta^\star, 1)$.  Then we have 
\[ h_{\tilde x}^2(\beta^\star, \beta) \leq 2\,k_{\tilde x}(\beta^\star, \beta) = |\tilde x^\top (\beta-\beta^\star)|^2. \]  
From the above inequality, upon taking expectation over $\tilde x \sim Q_n$, we get  
%\begin{equation}
%\label{eq:hellinger.bound}
\[ h^2(\beta^\star, \beta) \leq n^{-1} \|X(\beta-\beta^\star)\|^2. \]
%\end{equation}
Therefore, 
\[ \{\beta: h^2(\beta^\star, \beta) > M^2 \eps_n^2 \} \subseteq \{\beta: \|X(\beta-\beta^\star)\|^2 > M^2 n \eps_n^2\}, \]
and it follows from Theorem~1 in \citet{martin.mess.walker.eb}---see Appendix~\ref{S:bernoulli}---that, for suitable $M$, the expected $\Pi^n$-probability of the right-most event above is exponentially small, uniformly in $s^\star$-sparse $\beta^\star$.  
\end{proof}

The conditions here are different from those in \citet{jiang2007} and \citet{dunson.compreg.2015}, but it may help to compare the rates obtained.  Note that, beyond sparsity, no assumptions are made in Theorem~\ref{thm:hellinger.rate} above on the magnitude of $\beta^\star$, whereas the latter two papers assume $\|\beta^\star\|_1 = O(1)$ which requires either (a)~$s^\star$ grows slowly and non-zero signals vanish slowly, or (b)~$s^\star$ grows not-so-slowly and the non-zero signals vanish rapidly.  The more realistic case is (a), so suppose $s^\star \asymp (\log n)^k$ for some $k > 0$.  If $p$ is polynomial in $n$, i.e., $p \asymp n^K$ for any $K > 0$, then we have 
\[ \eps_n \asymp n^{-1/2} (\log n)^{(k+1)/2}, \]
which is nearly the parametric root-$n$ rate.  And if $p$ is sub-exponential in $n$, i.e., if $\log p \asymp n^r$ for $r \in (0,1)$, then $\eps_n$ is $n^{-(1-r)/2}$ modulo logarithmic terms which, again, is close to the parametric rate when $r$ is small.  In any case, if the analogy between the sparse normal means problem and the regression problem considered here holds up in the context of prediction, the minimax rate results in \citet{mukherjee.johnstone.2015} suggest that the rate $\eps_n$  in Theorem~\ref{thm:hellinger.rate} cannot be significantly improved.  Of course, posterior concentration rates in terms of $\|X(\beta-\beta^\star)\|$ have been established for other models, such as horseshoe \citep[e.g.,][]{ghosh.chakrabarti.shrink, pas.szabo.vaart.rate, pas.kleijn.vaart.2014}, so results like that in Theorem~\ref{thm:hellinger.rate} would apply for those methods as well.  

The next result connects the rate in Theorem~\ref{thm:hellinger.rate} to the posterior predictive density $f_x^n(\tilde y)$ defined in \eqref{eq:pred}.  Following the discussion above, this predictive density convergence rate is very fast, close to the root-$n$ rate in some cases.  

\begin{cor}
\label{thm:predictive.rate}
Let $f_{\tilde x}^\star(\tilde y) = \nm(\tilde y \mid \tilde x^\top \beta^\star, \sigma^2)$ denote the true distribution of the new observation, for a given $x$, and let $H(f_{\tilde x}^\star, f_{\tilde x}^n)$ denote the Hellinger distance between this and the predictive density $f_{\tilde x}^n$ in \eqref{eq:pred}.  Under the conditions of Theorem~\ref{thm:hellinger.rate}, 
\[ \sup_{\beta^\star: |S_{\beta^\star}|=s^\star} \E_{\beta^\star} \int H^2(f_{\tilde x}^\star, f_{\tilde x}^n) \, Q_n(d\tilde x) \lesssim \eps_n^2. \]
\end{cor}

\begin{proof}
Since $f_{\tilde x}^n(\tilde y) = \int \nm(\tilde y \mid \tilde x^\top \beta, \sigma^2) \, \Pi^n(d\beta)$, using convexity of $H^2$, Jensen's inequality, and Fubini's theorem we get  
\begin{align*}
\int H^2(f_{\tilde x}^\star, f_{\tilde x}^n) \, Q_n(d\tilde x) & \leq \int \int h_{\tilde x}^2(\beta^\star, \beta) \, \Pi^n(d\beta) \, Q_n(d\tilde x) \\
& = \int h^2(\beta^\star, \beta) \, \Pi^n(d\beta).
\end{align*}
For $M$ and $\eps_n$ as in Theorem~\ref{thm:hellinger.rate}, if we set $A_n = \{\beta: h(\beta^\star, \beta) \leq M \eps_n\}$, then the right-hand side above equals 
\begin{equation}
\label{eq:hbound2}
\int_{A_n} h^2(\beta^\star, \beta) \, \Pi^n(d\beta) + \int_{A_n^c} h^2(\beta^\star, \beta) \, \Pi^n(d\beta). 
\end{equation}
The first term is bounded by a constant times $\eps_n^2$ by definition of $A_n$.  And since Hellinger distance is no more than 2, the second term is bounded by a constant times $\Pi^n(A_n^c)$.  Theorem~\ref{thm:hellinger.rate} above shows that $\E_{\beta^\star} \Pi^n(A_n^c)$ is exponentially small, definitely smaller than $\eps_n^2$.  The claim follows since both terms in \eqref{eq:hbound2} are of order $\eps_n^2$ or smaller.  
\end{proof}

The particular metric in Theorem~\ref{thm:predictive.rate} measures the prediction quality for new $\tilde x$'s which are already in the current sample, that is, averaging over $\tilde x \sim Q_n$.  In other words, Theorem~\ref{thm:predictive.rate} considers a sort of in-sample prediction quality.  Intuitively, however, we expect that similar conclusions could be made for out-of-sample predictions, provided that the new $x$ value does not differ too much from the rows in the observed $X$.  Indeed, we are considering large $n$ so, if the rows of $X$ are sampled from a distribution $Q^\star$, then we can expect $Q_n$ to be a decent approximation of $Q^\star$, so results that involve averaging over $Q_n$ cannot be drastically different from the corresponding results with averaging over $Q^\star$.  Our simulation experiments investigate exactly this situation---new $\tilde x$ is an independent sample from the distribution that generated the original $X$---and the mean square prediction error results confirm the above intuition.

Formal out-of-sample prediction results are possible, but require additional assumptions about the design matrix.  Here we revert back to the general case of $d$-dimensional prediction, with $d \geq 1$, characterized by a $d \times p$ matrix $\tilde X$ at which prediction is desired.  Intuitively, if $\tilde X$ is genuinely new, then we have no direct measurements of $\tilde X \beta^\star$ as we would for in-sample prediction, so we cannot hope for quality prediction without accurate estimation of $\beta^\star$.  Since the response only provides direct information about the mean $X\beta^\star$, estimation of $\beta^\star$ requires disentangling $\beta^\star$ from $X\beta^\star$.  Towards this, following \citet[][Def.~2.3]{castillo.schmidt.vaart.reg} and \citet[][Eq.~11]{ariascastro.lounici.2014}, define the ``smallest scaled sparse singular value of $X$ of size $s$'' as 
\begin{equation}
\label{eq:kappa}
\lkappa(s; X) = n^{-1/2} \inf_{\beta \in \RR^p: 0 < |S_\beta| \leq s} \frac{\|X\beta\|}{\|\beta\|}, \quad s=0,1,\ldots,p. 
\end{equation}
\citet{ariascastro.lounici.2014} showed that a sparse $\beta^\star$, with $|S_{\beta^\star}|=s^\star$, is identifiable from a model with design matrix $X$ if and only if $\lkappa(2s^\star; X) > 0$.  So one must assume at least that in order to estimate $\beta^\star$ and to accurately predict at an arbitrary $\tilde X$.  

In Theorem~\ref{thm:oos.rate} below, we will need a complementary ``largest sparse singular value'' for the $d \times p$ matrix at which prediction is desired.  In particular, set 
\[ \ukappa(s; \tilde X) = d^{-1/2} \sup_{\beta \in \RR^p: 0 < |S_\beta| \leq s} \frac{\|\tilde X \beta\|}{\|\beta\|}, \quad s=0,1,\ldots,p. \]
The theorem and its corollary characterize the posterior and predictive distribution concentration rates, respectively, in terms of the $\lkappa$ and $\ukappa$ quantities.  Following their statements, we describe what these rates look like in some relevant cases.   

\begin{thm}
\label{thm:oos.rate}
Let $\Pi^n$ be the empirical Bayes posterior defined in \eqref{eq:update} and $s^\star$ a sequence such that \eqref{eq:standing} holds and $\lkappa(C s^\star; X) > 0$ for $C > 2$.  For fixed $d \geq 1$, let $\tilde X$ be a $d \times p$ matrix at which prediction is desired.  Then there exists positive constants $(C', G, M)$, with $C' > 2$, such that 
\[ \sup_{\beta^\star: |S_{\beta^\star}|=s^\star} \E_{\beta^\star} \Pi^n\Bigl(\Bigl\{\beta \in \RR^p: h_x(\beta^\star, \beta) > \frac{M \ukappa(C' s^\star; \tilde X) \eps_n}{\lkappa(C s^\star; X)} \Bigr\}\Bigr) \lesssim e^{-G n \eps_n^2} \to 0, \]
where $\eps_n^2 = n^{-1} s^\star \log (p/s^\star) \to 0$.  
\end{thm}

\begin{proof}
Again, without loss of generality, we assume $\sigma^2=1$.  Set 
\[ \Delta_n = \frac{M \ukappa(C' s^\star; \tilde X) \eps_n}{\lkappa(C s^\star; X)}. \]
From the total probability formula, the $\Pi^n$-probability in question is bounded by 
\[ \Pi^n(\{\beta \in \RR^p: h_x(\beta^\star, \beta) > \Delta_n, |S_\beta| \leq C'' s^\star\}) + \Pi^n(\{\beta \in \RR^p: |S_\beta| > C''s^\star\}). \]
Theorem~2 in \citet{martin.mess.walker.eb} shows that, for some $C'' > 1$, the expected value of the second term in the above display is exponentially small; see Appendix~\ref{S:bernoulli} for a precise statement.  So it suffices to focus on the first term.  If $\beta$ is such that $|S_\beta| \leq C''s^\star$, then $|S_{\beta-\beta^\star}| \leq C' s^\star$, where $C' = C''+1 > 2$.  So by definition of $\ukappa$, 
\[ h_{\tilde X}^2(\beta^\star, \beta) \leq \|\tilde X (\beta-\beta^\star)\|^2 \leq \ukappa^2(C's^\star; \tilde X) \|\beta-\beta^\star\|^2. \]
Therefore, 
\[ \{h_{\tilde X}(\beta^\star, \beta) > \Delta_n\} \cap \{|S_\beta| \leq C''s^\star\} \implies \|\beta-\beta^\star\| > \frac{M \eps_n}{\lkappa(C s^\star; X)}, \]
and Theorem~3 in \citet{martin.mess.walker.eb}---see Appendix~\ref{S:bernoulli}---establishes that the expected $\Pi^n$-probability of the latter event is exponentially small.
\end{proof}

Like in the in-sample prediction setting, this posterior concentration rate result in Theorem~\ref{thm:oos.rate} for out-of-sample prediction can be immediately converted into a rate for the predictive density $f_x^n(\tilde y)$ defined in \eqref{eq:pred} at the given $x$.  

\begin{cor}
\label{cor:oos.rate}
For a fixed $d \times p$ matrix $\tilde X$ at which prediction is desired, let $f_{\tilde X}^\star(\tilde y) = \nm_d(\tilde y \mid \tilde X^\top \beta^\star, \sigma^2 I_d)$ denote the true distribution of the new observation and $H(f_{\tilde X}^\star, f_{\tilde X}^n)$ the Hellinger distance between this and the predictive density $f_{\tilde X}^n$ in \eqref{eq:pred}.  Under the assumptions of Theorem~\ref{thm:oos.rate}, 
%If $s^\star = o(n)$, then  
\[ \sup_{\beta^\star: |S_{\beta^\star}|=s^\star} \E_{\beta^\star} H^2(f_{\tilde X}^\star, f_{\tilde X}^n) \lesssim \frac{\ukappa^2(C' s^\star; \tilde X) \eps_n^2}{\lkappa^2(C s^\star; X)}. \]
%where $\eps_n^2 = n^{-1} s^\star \log p$.  With more control on the sparsity level, in particular, $s^\star = O(n^\delta)$ for some $\delta \in (0,1)$, then $\eps_n^2 = n^{-1} s^\star \log(p / s^\star)$. 
\end{cor}

\begin{proof}
The proof proceeds exactly like that of Corollary~\ref{thm:predictive.rate} but for not needing to integrate over $\tilde X$ and applying Theorem~\ref{thm:oos.rate} instead of Theorem~\ref{thm:hellinger.rate}.  
\end{proof}

Of course, the $\eps_n$ in the above two results is the same as in Theorem~\ref{thm:hellinger.rate} and its corollary, so can still be close to the parametric root-$n$ rate.  However, the rate in these latter results depends in a non-trivial way on the design matrix $X$ and the new $\tilde X$ at which prediction is desired, through the $\lkappa$ and $\ukappa$ quantities.  A relevant case to consider is that where the rows of $X$ and of $\tilde X$ are iid Gaussian $p$-vectors with mean 0, variance 1, and correlation matrix $\Sigma$.  For a number of relevant correlation structures in $\Sigma$, e.g., independence, first-order autoregressive, block diagonal, etc., it is known that $\lkappa(s; X)$ is bounded away from 0, with high probability (as a function of the random $X$), for all $s \lesssim (n/\log p)^{1/2}$; see \citet[][Examples~7--8]{castillo.schmidt.vaart.reg} and \citet[][Sec.~2.6]{ariascastro.lounici.2014}.  So, at least for these cases, the $\lkappa$ term in the denominator is not expected to affect the rate in Theorem~\ref{thm:oos.rate} or its corollary.  Next, for the numerator, we can bound $\ukappa(s; \tilde X) \leq d^{-1/2} \|\tilde X_S\|_F$, where $S$ is a subset of size $|S| \leq s$ and $\|\cdot\|_F$ denotes the Frobenius norm.  And since the entries of $\tilde X$ are $O_p(1)$, we can conclude that $\ukappa(s; \tilde X) = O_p(s^{1/2})$.  Therefore, putting all this together, the upper bound obtained in Corollary~\ref{cor:oos.rate} is $\lesssim s^\star \eps_n^2$.  And for small $s^\star$, e.g., a power of $\log n$ like discussed above, the overall rate is still roughly root-$n$.  To put this ``$s^\star \eps_n^2$'' quantity in perspective, consider an oracle case where the configuration $S^\star$, of size $s^\star$, is known.  Then the mean square error for the oracle least squares estimator of the mean response, $\tilde X_{S^\star} \hat \beta_{S^\star}$, given $X$ and $\tilde X$, is 
\[ \E_{\beta^\star}\|\tilde X_{S^\star} \hat\beta_{S^\star} - \tilde X_{S^\star} \beta_{S^\star}^\star\|^2 = \text{tr}\{\tilde X_{S^\star} (X_{S^\star}^\top X_{S^\star})^{-1} \tilde X_{S^\star}^\top\}. \]
Plugging in the approximation, $n^{-1} X_{S^\star}^\top X_{S^\star} \approx \Sigma_{S^\star}$, where the latter is the submatrix of $\Sigma$ corresponding to the variables identified in $S^\star$, it is easy to check that the trace is $n^{-1}$ times a chi-square random variable with $ds^\star$ degrees of freedom, i.e., is $O_p(s^\star n^{-1})$, as a function of $(X,\tilde X)$.  Therefore, since the $s^\star \eps_n^2$ upper bound in Corollary~\ref{cor:oos.rate} is of the same order as the oracle mean square error for estimating the mean response at $\tilde X$, our rates cannot be significantly improved.

\subsection{Uncertainty quantification}
\label{SSS:uq}

Beyond prediction accuracy, one would also want the posterior predictive distribution to be calibrated in the sense that a $100(1-\zeta)$\% prediction interval, for $\zeta \in (0,\frac12)$, has coverage probability $1-\zeta$, at least approximately.  That is, one may ask if the predictive distribution above provides valid uncertainty quantification.  To first build some intuition, recall the predictive density in \eqref{eq:pred}:
\[ f_{\tilde X}^n(\tilde y) = \sum_S \pi^n(S) \, \nm_d\bigl(\tilde y \mid \tilde X_S \hat\beta_S, \sigma^2 I_d + \tfrac{\sigma^2}{\alpha + \gamma} \tilde X_S (X_S^\top X_S)^{-1} \tilde X_S^\top \bigr). \]
If we happen to have $\E_{\beta^\star}\pi^n(S^\star) \to 1$, then 
\begin{equation}
\label{eq:pred.approx}
f_{\tilde X}^n(\tilde y) \approx \nm_d\bigl(\tilde y \mid \tilde X_{S^\star} \hat\beta_{S^\star}, \sigma^2 I_d + \tfrac{\sigma^2}{\alpha + \gamma} \tilde X_{S^\star} (X_{S^\star}^\top X_{S^\star})^{-1} \tilde X_{S^\star}^\top \bigr), 
\end{equation}
and one will recognize the right-hand side as roughly the {\em oracle} predictive distribution, the one based on knowledge of the correct configuration $S^\star$.  The only difference between this predictive distribution and the standard fixed-model version found in textbooks is the factor $(\alpha + \gamma)^{-1}$.  We prefer our predictive density to be at least as wide as the oracle, which suggests choosing $(\alpha,\gamma)$ such that $\alpha + \gamma \leq 1$, maybe strictly less than 1.  With this choice, we expect the posterior prediction intervals to approximately achieve the nominal frequentist coverage probability. %, and our numerical results in Section~\ref{S:examples} confirm this.  

%In our examples we also observe similarly good coverage for prediction intervals derived from the model that starts with the horseshoe prior, which is expected based on the results in \citet{pas.szabo.vaart.uq} for the normal means model.  However, to our knowledge, there are is no theoretical work in the literature that establishes the validity of horseshoe-based prediction intervals in the regression setting considered here.  

To make the above heuristics precise, write the posterior distribution for $\beta$ as 
\[ \Pi^n(B) = \sum_S \pi^n(S) \, \bigl\{ \nm_{|S|}\bigl(\hat\beta_S, \tfrac{\sigma^2}{\alpha + \gamma} (X_S^\top X_S)^{-1} \bigr) \otimes \delta_{0_{S^c}} \bigr\}(B), \quad B \subseteq \RR^p. \]
For the given $d \times p$ matrix $\tilde X$ at which prediction is desired, set $\psi = \tilde X \beta$.  Then the derived posterior distribution for $\psi$ is 
\begin{equation}
\label{eq:post.psi}
\Pi_\psi^n(A) := \sum_S \pi^n(S) \, \nm_d\bigl(A \mid \hat\psi_S, \sigma^2 V_S \bigr), \quad A \subseteq \RR^d, 
\end{equation}
where 
\[ \hat\psi_S = \tilde X_S \hat\beta_S \quad \text{and} \quad V_S = (\alpha + \gamma)^{-1} \tilde X_S (X_S^\top X_S)^{-1} \tilde X_S^\top. \]
Then we have the following Bernstein--von Mises theorem, similar to that in \citet{ebcvg}, which formally establishes a Gaussian approximation to the marginal posterior distribution of $\psi$, which will almost immediately justify \eqref{eq:pred.approx}.

\begin{thm}
\label{thm:bvm}
Write $d_{\text{\sc tv}}(P,Q) = \sup_A |P(A) - Q(A)|$ for the total variation distance between probability measures $P$ and $Q$.  Let $\nm_d(\hat\psi_{S^\star}, \sigma^2 V_{S^\star})$ denote the oracle posterior for $\psi$ based on knowledge of the true configuration $S^\star$.  If $\E_{\beta^\star} \pi^n(S^\star) \to 1$, then 
\[ \E_{\beta^\star} d_{\text{\sc tv}}\bigl( \Pi_\psi^n, \nm_d(\hat\psi_{S^\star}, \sigma^2 V_{S^\star}) \bigr) \to 0. \]
\end{thm}

\begin{proof}
Define $D_n(A) = | \Pi_\psi^n(A) - \nm_d(A \mid \hat\psi_{S^\star}, \sigma^2 V_{S^\star}) |$ for Borel sets $A \subseteq \RR^d$.  Since $|\sum_i a_i| \leq \sum_i |a_i|$, we immediately get the following upper bound:
\[ D_n(A) \leq \sum_S \pi^n(S) \bigl| \nm_d(A \mid \hat\psi_S, \sigma^2 V_S) - \nm_d(A \mid \hat\psi_{S^\star}, \sigma^2 V_{S^\star}) \bigr|. \]
The absolute difference above is 0 when $S=S^\star$ and bounded by 2 otherwise, so 
\[ d_{\text{\sc tv}}\bigl( \Pi_\psi^n, \nm_d(\hat\psi_{S^\star}, \sigma^2 V_{S^\star}) \bigr) \leq 2 \sum_{S \neq S^\star} \pi^n(S) = 2 \{1 - \pi^n(S^\star)\}. \]
After taking expectation of both sides, the upper bound vanishes by assumption.  
\end{proof}

Our Bernstein--von Mises theorem enjoys a relatively simple proof thanks to two things: the conjugate prior and the assumption that $\E_{\beta^\star} \pi^n(S^\star) \to 1$.  First, as a result of using a conjugate  (empirical) prior for $\beta_S$, given $S$, the marginal posterior distribution for $\psi$ is exactly a mixture of Gaussians as in \eqref{eq:post.psi}.  Therefore, if the posterior for $S$ concentrates on $S^\star$, then the result follows immediately.  If not for this conjugacy, e.g., if one used the Laplace-type priors from \citet{castillo.schmidt.vaart.reg}, then a proof of result like that in Theorem~\ref{thm:bvm} would be much more involved, even with the posterior for $S$ concentrating on $S^\star$.  Second, the ``selection consistency'' property, namely, $\E_{\beta^\star} \pi^n(S^\star) \to 1$, is one that has been extensively studied in the Bayesian literature on high-dimensional regression.  Beyond the conditions needed for the rate results in Section~\ref{SSS:rate}, to be able to detect the correct set of variables, the non-zero coefficients need to be sufficiently large.  More formally, these considerations lead to a so-called ``beta-min condition'' \citep[e.g.,][]{ariascastro.lounici.2014, buhlmann.geer.book, castillo.schmidt.vaart.reg}, namely, \begin{equation}
\label{eq:beta.min}
\min_{j \in S_{\beta^\star}} |\beta_j^\star| \gtrsim \Bigl\{ \frac{\sigma^2 \log p}{n \lkappa^2(C|S_{\beta^\star}|; X)} \Bigr\}^{1/2}. 
\end{equation}
Note that the definition of ``sufficiently large'' here depends on $\lkappa$ and, therefore, implicitly assumes that this quantity is positive.  Under the beta-min condition, \citet{martin.mess.walker.eb} establish selection consistency for the empirical Bayes posterior $\Pi^n$ under investigation here; see Appendix~\ref{S:bernoulli} for a precise statement.  It is intuitively clear that a condition like \eqref{eq:beta.min} is needed in order to identify the correct $S^\star$, and it is apparently not too strong of an assumption given that existing methods are able to recover $S^\star$ empirically for a wide range of examples.  We will have more to say about the beta-min condition in the context of our main result of this section, Corollary~\ref{cor:cvg} below, on the coverage probability of credible intervals derived from the posterior predictive distribution.   

The Gaussian approximation to the marginal posterior $\Pi_\psi^n$ should not be a big surprise, given its relatively simple form in \eqref{eq:post.psi} and the strong selection consistency property discussed above.  Of critical importance for us here is that Theorem~\ref{thm:bvm} implies a corresponding Bernstein--von Mises theorem for the predictive density.   

\begin{cor}
\label{cor:bvm}
Under the conditions of Theorem~\ref{thm:bvm}, 
%holds for the predictive distributions.  That is, 
%If $\E_{\beta^\star}\pi^n(S^\star) \to 1$, then 

\[ \E_{\beta^\star} d_{\text{\sc tv}} \bigl( f_{\tilde X}^n, \nm_d(\hat\psi_{S^\star}, \sigma^2(I_d + V_{S^\star}) \bigr) \to 0. \]
\end{cor}

\begin{proof}
The posterior and oracle predictive distributions are convolutions of $\nm_d(0,\sigma^2I_d)$ with $\Pi_\psi^n$ and $\nm_d(\hat\psi_{S^\star}, \sigma^2 V_{S^\star})$, respectively.  
%The predictive distribution is a convolution of $\Pi_\psi^n$ with $\nm(0,\sigma^2)$ and, similarly, the oracle predictive distribution is a convolution of $\nm(\hat\psi_{S^\star}, \sigma^2 v_{S^\star})$ with $\nm(0,\sigma^2)$.  
So the claim follows from Theorem~\ref{thm:bvm} and general results on information loss, namely, Lemma~B.11 and Equation~(B.14) in \citet{ghosal.vaart.book}.  
\end{proof}

The conclusion of Corollary~\ref{cor:bvm} is that the predictive distribution will closely resemble the oracle predictive distribution based on knowledge of $S^\star$.  To visualize this, we carry out a small simulation study where data $y$ is generated from model \eqref{eq:reg}, with $\sigma=1$, under the ``very-sparse'' settings described in Section~\ref{SSS:very.sparse} below.  In particular, we have $n=100$ and $p=250$, but there are only $|S^\star|=5$ non-zero coefficients, each of magnitude $A=4$.  The $X$ matrix has rows that are iid $p$-variate normal with mean 0, variance 1, and a first-order autoregressive correlation structure with parameter $r=0.5$.  The goal is to predict a new $\tilde y$ of dimension $d=1$, where the corresponding $\tilde x$ vector is an independent draw from the same $p$-variate normal distribution that generated $X$.  Figure~\ref{fig:pred} shows a histogram of a sample drawn from our posterior predictive distribution with the corresponding oracle predictive density function overlaid.  These computations were actually done in an unknown-$\sigma^2$ scenario, so that the oracle is a shifted and scaled Student-t density.  That the two distributions match very closely confirms Corollary~\ref{cor:bvm}, but it is perhaps surprising that the accuracy of the normal approximation kicks in with only $n=100$, even for a relatively high-dimensional setting.  

\begin{figure}[t]
\begin{center}\scalebox{0.7}{\includegraphics{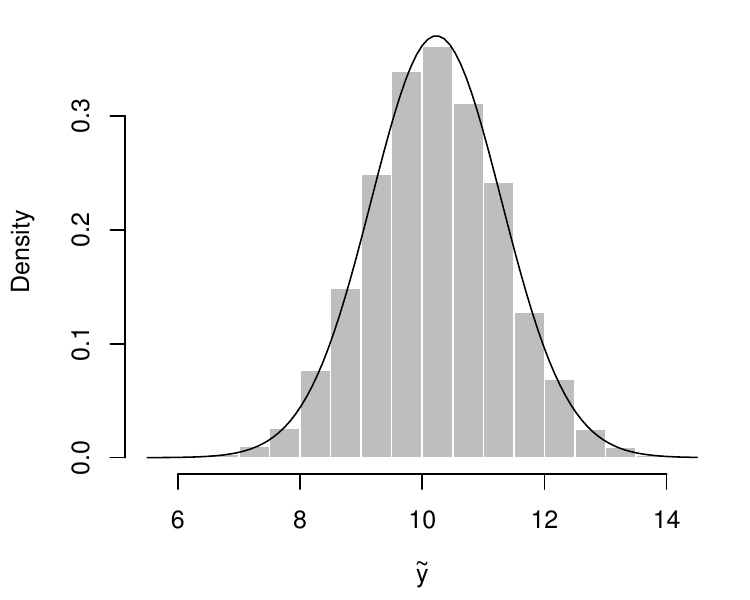}}
\end{center}
\caption{Histogram of a Monte Carlo sample drawn from the posterior predictive distribution, $f_x^n$ in \eqref{eq:pred}, and the corresponding oracle predictive density function overlaid.}
\label{fig:pred}
\end{figure}

%From this, we can provide a rigorous justification for the claim that prediction intervals derived from the proposed empirical Bayes predictive distribution asymptotically achieve

The Bernstein--von Mises result in Corollary~\ref{cor:bvm} is in terms of a strong, total variation distance, which implies that central $1-\zeta$ regions derived from the posterior predictive distribution are approximately 
\[ \bigl\{\tilde y \in \RR^d: (\tilde y - \hat\psi_{S^\star})^\top (I_d + V_{S^\star})^{-1} (\tilde y - \hat\psi_{S^\star}) \leq \sigma^2 \chi_{\zeta/2}^2(d) \bigr\}, \]
where $\chi_{\zeta/2}^2(d)$ is the upper-$\zeta/2$ quantile of $\chisq(d)$.  For the special case $d=1$, and in the original notation, the posterior prediction interval for $\tilde y$, at a covariate (column) vector $\tilde x$, is approximately 
\[ x^\top \hat \beta_{S^\star} \pm z_{\zeta/2} \, \sigma \bigl\{ 1 + (\alpha + \gamma)^{-1} \tilde x_{S^\star}^\top (X_{S^\star}^\top X_{S^\star})^{-1} \tilde x_{S^\star} \bigl\}^{1/2}, \]
where $z_{\zeta/2}$ is the upper-$\zeta/2$ quantile of the standard normal distribution.  Clearly, this is a $100(1-\zeta)$\% frequentist prediction interval for a new $\tilde y$ drawn from model \eqref{eq:reg}, with the new covariate vector $\tilde x$, provided that $\alpha + \gamma \leq 1$.  Moreover, as this limiting credible interval matches that of an oracle who knows $S^\star$, the size cannot be improved.  

\begin{cor}
\label{cor:cvg}
Under the conditions of Theorem~\ref{thm:bvm}, if $\alpha + \gamma \leq 1$, then prediction intervals for a new $\tilde y$ drawn from model \eqref{eq:reg}, with the new covariate vector $\tilde x \in \RR^p$, derived from the empirical Bayes posterior predictive distribution \eqref{eq:pred} attain the nominal frequentist coverage probability and are of optimal size asymptotically.  
\end{cor}

To conclude this section, we give a few remarks to put the result in Corollary~\ref{cor:cvg} in perspective.  First, to our knowledge, there are no uncertainty quantification results for the Bayes solution in the high-dimensional regression setting based on the horseshoe prior, although the developments in \citet{pas.szabo.vaart.uq} would suggest that results along this line are possible under appropriate conditions.  More recently,  \citet{belitser.ghosal.ebuq} investigated the uncertainty quantification properties of an empirical Bayes solution closely related to that in \citet{martin.mess.walker.eb} adopted here, but not in the context of prediction.  Roughly, they show that suitably constructed empirical Bayes credible balls for the full $\beta$ vector approximately achieve the nominal frequentist coverage probability.  For example, in their Corollary~4, for any credibility level $\zeta$, they identify an $\ell_2$-norm ball 
\[ B_n(\zeta) = \{\beta \in \RR^p: \|\beta-\hat\beta\| \leq M_\zeta \, \hat\rho_n^\star\}, \]
where $\hat\beta$ is the empirical Bayes posterior mean, $\hat\rho_n^\star$ is a function depending on data, on features of $\beta^\star$, and on other quantities, and $M_\zeta$ is a constant, such that 
\[ \sup_{\beta^\star} \prob_{\beta^\star}\{ B_n(\zeta) \not\ni \beta^\star\} \leq \zeta, \]
with the supremum taken over a class of vectors $\beta^\star$ that satisfy an {\em excessive bias restriction}.  The specific details are not relevant for the present discussion, so suffice it to say that the excessive bias restriction rules about $\beta^\star$ are difficult to detect, such as sparse vectors whose non-zero entries are relatively small.  Therefore, sparsity plus the beta-min condition \eqref{eq:beta.min} implies the excessive bias restriction.  Consequently, if $\tilde x$ is a new $p$-vector at which prediction is required, one can naively convert the above credible ball into a credible interval for the mean response $\mu=\mu(\tilde x)$ at $\tilde x$, i.e., 
\[ \{\mu \in \RR: |\mu - \tilde x^\top \hat\beta| \leq M_\zeta \hat\rho_n^\star / \|\tilde x\|\}, \]
and this, in turn, can be suitably enlarged for predicting an independent response at $\tilde x$.  Although this prediction interval would asymptotically achieve the target coverage probability under the weaker excessive bias restriction, it still has some theoretical and practical disadvantages.  First, a number of the ingredients in the credible ball $B_n(\zeta)$ and in the corresponding prediction interval are either unspecified (e.g., depend on ``sufficiently large'' constants) or depend on features of the unknown $\beta^\star$, so putting the theory into practice is not straightforward.  Second, computation of this alternative prediction interval is more expensive than our proposed approach based on \eqref{eq:pred} because the high-dimensional $\beta$ cannot be integrated out.  Third, because of their indirect/naive construction, it is likely that the prediction intervals described above are less efficient than those in Corollary~\ref{cor:cvg}, which are optimal.  So even though our uncertainty quantification results here make stronger assumptions than the weakest of those appearing in other contexts in the literature, our conclusions are stronger and our implementation is easier.  Plus, our simulation experiments show strong empirical performance across various settings.

\section{Numerical results}
\label{S:examples}

\subsection{Methods}

Here we investigate the performance of our empirical Bayes prediction method compared to existing methods.  Our method, which we denote by {\em EB}, is as described in Section~\ref{S:predictive}, with the following hyperparameter settings:
\begin{itemize}
\item the complexity prior $q_n$ for the size $|S|$ in \eqref{eq:complexity} uses $a=0.05$ and $c=1$; 
\vspace{-2mm}
\item the posterior construction in Section~\ref{S:review} uses $\gamma=0.005$ in \eqref{eq:beta.prior.S} and $\alpha=0.99$ in \eqref{eq:update}; 
\vspace{-2mm}
\item and the inverse gamma prior for $\sigma^2$ described in Section~\ref{SS:unknown} has shape and scale parameters $a_0=0.01$ and $b_0=4$, respectively. 
\end{itemize}
%{\color{red} R code to implement EB is available at \url{https://www4.stat.ncsu.edu/~rmartin/research.html}.}  
The latter two are the default settings in the R package {\tt ebreg}.  This is compared to predictions based on the horseshoe, denoted by {\em HS}, using the default Jeffreys prior for $\sigma^2$ implemented in the {\tt horseshoe} package \citep{horseshoe.package}.  That package does not return samples from the predictive distribution, but these are easy to obtain from the $(\beta,\sigma^2)$ samples.  In our simulations, both the EB and HS methods return $5000$ posterior predictive samples after a burn-in of 1000.   The aforementioned methods give full predictive distributions, which yield both point predictions and prediction intervals.  We also compare with point predictions obtained from lasso and adaptive lasso using the R packages {\tt lars} \citep{lars} and {\tt parcor} \citep{parcor}, respectively.

\subsection{Simulated data experiments}

The theory presented in Section~\ref{SS:theory} focuses on the high- but not ultra-high-dimensional setting described in \citet{verzelen2012}, i.e., where $(n,p,s^\star)$ are such that $s^\star \log(p / s^\star)$ is small compared to $n$.  So in our examples presented below, we consider settings, common in medical and social science applications, with $p$ ranging from just slightly larger than $n$ to several times larger than $n$.  To keep the focus on the effect of the ambient and effective dimensions, i.e., $p$ and $s^\star$, on prediction performance, throughout these experiments we take $n=100$ fixed, and vary $(p,s^\star)$ accordingly.  In limited runs with larger sample sizes, i.e., $n \in \{200, 300, 400\}$, we found that the relative comparisons between methods was comparable, so these results are not presented here.  We investigate the prediction error, the coverage rates, and length of prediction intervals produced by the proposed method compared to the above competitors.  In particular, we consider univariate predictions ($d=1$) in an ``out-of-sample'' context where the new $\tilde x$ vector is independently sampled from that same $p$-variate normal distribution described above, and the goal is to predict a new response $\tilde y = \tilde x^\top \beta^\star + \sigma z$ from the model \eqref{eq:reg}.  
 
The take-away message here is that EB is generally better than HS across a range of settings, in terms of both prediction accuracy and uncertainty quantification.  In particular, first, in the out-of-sample mean square prediction error comparisons, there are cases where the two methods have comparable performance and a number of cases where EB is significantly better, but no cases where HS is significantly better.  Second, the prediction intervals for the two methods both generally are within an acceptable range of the target 95\% prediction coverage probability but, with the exception of a few cases at $p=500$ in the very-sparse setting (Table~\ref{table:len}), the EB intervals tend to be shorter than the HS intervals.   On top of its strong statistical performance, our EB method is more efficient computationally, as it generally finishes 5 times faster than the implementation of HS in the corresponding R package.

\subsubsection{Very-sparse settings}
\label{SSS:very.sparse}

We select $s^\star=5$ specific $\beta_j$'s to be non-zero, with the rest being zero. In particular, we place non-zero values at positions 3, 4, 15, 22, and 25 in the $p$-vector $\beta$. This configuration captures a number of different features: 3 and 4 are adjacent, 4 and 15 have a large gap between them, and 22 and 25 are close neighbors but not adjacent.  All of the non-zero $\beta_j$'s take value $A$, where $A \in \{2,4,8\}$.  The rows of the design matrix, $X$, are $p$-variate normal with zero mean, unit variance, and first-order autoregressive dependence structure, with pairwise correlation equals $r^{|i-j|}$ and $r \in \{0.2, 0.5, 0.8\}$.  For each $(A,r)$ pair, we consider $p\in \{125, 250, 500\}$, which yields a total of 27 different settings.  MSPEs are shown in Table~\ref{table:mspe} and the prediction interval coverage probabilities and mean lengths are presented in Tables~\ref{table:covg} and \ref{table:len}, respectively.  Summaries of the standard errors are provided in the respective table captions, and based on 250 runs.  

\begin{table}[t]
{\small 
\begin{center}  
\begin{tabular}{ccccccccccc}
    \toprule
    \multirow{2}{*}{$p$} & %\multirow{2}{*}{Method} 
&  \multicolumn{3}{c}{$A=2$}& \multicolumn{3}{c}{$A=4$}& \multicolumn{3}{c}{$A=8$}\\
     & & $r=0.2$ & $0.5$ & $0.8$ & $r=0.2$ & $0.5$ & $0.8$ & $r=0.2$ & $0.5$ & $0.8$\\
     \midrule
    125 & EB & 0.86 & 1.13 & 0.99 & 1.15 & 0.91 & 0.94 & 0.98 & 0.99 & 1.22\\  
     & HS & 0.89 & 1.21 & 1.07 & 1.20 & 1.01 & 1.05 & 1.03 & 1.07 & 1.31\\
     & Lasso & 1.18 & 1.51 & 1.38 & 1.57 & 1.48 & 1.48 & 1.66 & 1.51 & 1.90\\
     & Alasso & 0.93 & 1.15 & 1.05 & 1.70 & 1.45 & 1.76 & 5.20 & 3.53 & 8.99\\
     \hline
    250  & EB & 1.05 & 1.10 & 1.12 & 1.07 & 1.24 & 0.87 & 1.15 & 1.07 & 0.96\\ 
     & HS & 1.12 & 1.19 & 1.17 & 1.12 & 1.34 & 0.98 & 1.20 & 1.12 & 1.00\\
     & Lasso & 1.31 & 1.50 & 1.52 & 1.38 & 1.69 & 1.35 & 1.31 & 1.35 & 1.33\\
     & Alasso & 1.07 & 1.14 & 1.23 & 1.50 & 1.92 & 2.07 & 3.47 & 4.14 & 5.60\\
    \hline
    500 &EB & 0.97 & 0.93 & 0.93 & 0.93 & 1.16 & 1.08 & 1.01 & 1.02 & 1.17\\ 
     & HS & 1.00 & 1.08 & 0.99 & 1.09 & 1.15 & 1.05 & 1.12 & 1.04 & 1.19\\
     & Lasso & 1.28 & 1.37 & 1.44 & 1.49 & 1.35 & 1.25 & 1.43 & 1.55 & 1.44\\
     & Alasso & 0.98 & 1.00 & 1.12 & 1.35 & 1.44 & 2.14 & 4.54 & 3.83 & 6.52\\
    \bottomrule
\end{tabular}
\end{center}
}
\caption{Comparison of mean square prediction error (MSPE) for the four different methods across various very sparse settings---of dimension $p \in \{125, 250, 500\}$, signal size $A \in \{2,4,8\}$, and correlation $r \in \{0.2, 0.5, 0.8\}$---as described in the text. The standard errors are generally between 0.08 and 0.98.
%, with the exception of one large standard error of 2.54. 
}
\label{table:mspe}
\end{table}

EB performs very well across all the settings in terms of MSPE, with HS performing similarly, and both generally beating lasso and adaptive lasso. The standard errors generally range between 0.08 and 0.98, with an average of 0.18. Based on these standard errors, there are no significant differences in prediction performance between EB and HS, at least not in terms of MSPE.  However, there is a substantial difference in terms of computation time.  EB is based on a two-groups or spike-and-slab model formulation, generally believed to be too expensive to compute.  But contrary to this popular belief, EB's run-time is about 20\% of that for HS, consistent with the claims made by \citet{ebcvg} in a different context.   

%EB and HS are equally good methods in these settings, and exhibit no significant differences in their MSPEs. 

%We also took a look at the distribution of the 250 runs for each setting, and note that the maximum of the squared prediction errors are generally around 10. 

%We also found that in our model settings, the Empirical Bayes MSPE standard deviation is consistently smaller than that of the other methods. This seems to show that our method is more stable in these settings than the other methods. We also note that even though in these estimates, our MSPE estimates are better in many of the settings, we cannot conclude statistically that our MSPE is better than the methods considering standard deviation. 

For the 95\% prediction interval comparisons, we compare EB and HS to that of an {\em oracle} who knows the true configuration $S^\star$.  Of course, the coverage probabilities for the oracle prediction interval are exactly 0.95, but, according to Table~\ref{table:covg}, both EB and HS have prediction coverage probability within an acceptable range of the target 95\% level.  And in terms of interval lengths, Table~\ref{table:len} reveals that both EB and HS are comparable in efficiency to the oracle prediction intervals.  This confirms the claims made about EB prediction intervals in Section~\ref{SS:theory}.   There are, however, a few cases in the $p=500$ row of Table~\ref{table:len} where HS gives significantly shorter interval lengths.  Obviously, the $p=500$ case is generally harder, but it is not clear if this performance difference is just a coincidence or if there is something special about these cases.  A theoretical investigation into the uncertainty quantification properties of the horseshoe prior clearly is needed.  

\begin{table}[t]
{\small 
\begin{center}  
\begin{tabular}{ccccccccccc}
    \toprule
    \multirow{2}{*}{$p$} & %\multirow{2}{*}{Method} 
&  \multicolumn{3}{c}{$A=2$}& \multicolumn{3}{c}{$A=4$}& \multicolumn{3}{c}{$A=8$}\\
     & & $r=0.2$ & $0.5$ & $0.8$ & $r=0.2$ & $0.5$ & $0.8$ & $r=0.2$ & $0.5$ & $0.8$\\
     \midrule
    125 & EB & 0.96 & 0.94 & 0.96 & 0.95 & 0.96 & 0.97 & 0.97 & 0.96 & 0.94\\  
     & HS & 0.96 & 0.92 & 0.96 & 0.93 & 0.96 & 0.94 & 0.96 & 0.94 & 0.92\\
     \hline
    250 & EB & 0.95 & 0.96 & 0.95 & 0.96 & 0.94 & 0.96 & 0.94 & 0.96 & 0.95\\ 
     & HS & 0.95 & 0.96 & 0.94 & 0.96 & 0.92 & 0.96 & 0.94 & 0.94 & 0.95\\
    \hline
    500 & EB & 0.95 & 0.95 & 0.96 & 0.96 & 0.94 & 0.96 & 0.96 & 0.93 & 0.94\\ 
     & HS & 0.95 & 0.94 & 0.95 & 0.93 & 0.94 & 0.96 & 0.95 & 0.92 & 0.93\\
    \bottomrule
\end{tabular}
\end{center}
}
\caption{Comparison of coverage probability for the two different 95\% prediction intervals across various very sparse settings---of dimension $p \in \{125, 250, 500\}$, signal size $A \in \{2,4,8\}$, and correlation $r \in \{0.2, 0.5, 0.8\}$---as described in the text. Standard errors are between 0.01 and 0.02. }
\label{table:covg}
\end{table}

\begin{table}[t]
{\small 
\begin{center}  
\begin{tabular}{ccccccccccc}
    \toprule
    \multirow{2}{*}{$p$} & 
&  \multicolumn{3}{c}{$A=2$}& \multicolumn{3}{c}{$A=4$}& \multicolumn{3}{c}{$A=8$}\\
     & & $r=0.2$ & $0.5$ & $0.8$ & $r=0.2$ & $0.5$ & $0.8$ & $r=0.2$ & $0.5$ & $0.8$\\
     \midrule
    125 & EB & 4.12 & 4.13 & 4.15 & 4.11 & 4.14 & 4.16 & 4.13 & 4.10 & 4.17\\  
     & HS & 4.15 & 4.16 & 4.17 & 4.15 & 4.18 & 4.16 & 4.15 & 4.13 & 4.16\\
     & Oracle & 4.06 & 4.07 & 4.07 & 4.05 & 4.09 & 4.08 & 4.07 & 4.04 & 4.07\\
     \hline
    250 & EB & 4.13 & 4.15 & 4.18 & 4.12 & 4.14 & 4.13 & 4.15 & 4.12 & 4.14\\ 
     & HS & 4.14 & 4.17 & 4.15 & 4.13 & 4.14 & 4.09 & 4.16 & 4.12 & 4.12\\
     & Oracle & 4.07 & 4.09 & 4.09 & 4.06 & 4.08 & 4.05 & 4.09 & 4.07 & 4.05\\
    \hline
    500 & EB & 4.12 & 4.14 & 4.17 & 4.08 & 4.15 & 4.15 & 4.10 & 4.12 & 4.18\\ 
     & HS & 4.10 & 4.08 & 4.11 & 4.05 & 4.09 & 4.07 & 4.05 & 4.09 & 4.10\\
     & Oracle & 4.07 & 4.08 & 4.08 & 4.04 & 4.08 & 4.05 & 4.04 & 4.06 & 4.09\\
    \bottomrule
\end{tabular}
\end{center}
}
\caption{Comparison of mean length for the three different 95\% prediction intervals across various very sparse settings---of dimension $p \in \{125, 250, 500\}$, signal size $A \in \{2,4,8\}$, and correlation $r \in \{0.2, 0.5, 0.8\}$---as described in the text. Standard errors are between 0.01 and 0.03.}
\label{table:len}
\end{table}

%Our empirical results are consistent with our theoretical derivations. Our prediction intervals have coverage roughly 95\%, and are on par with the coverage of the prediction intervals produced by the Horseshoe methods. The length of the EB prediction intervals are smaller than those of the horseshoe in most configurations. We do, however, note that in configurations with larger signal sizes, the EB1 prediction intervals are not ideal as they are longer in length. In the low signal size settings ($a=2$), EB1 produced prediction intervals that are either shorter than or very close to the length of the oracle; in larger signal size settings ($a \in \{4, 8\}$), the EB2 prediction intervals have length similar to those of the oracle. 

\subsubsection{Less-sparse settings}

Let $\Delta \in (0,1)$ control the density of the signals, so that $s^\star = \Delta p$.  Obviously, if the density is too large, then the problem will eventually turn ultra-high-dimensional, so we focus on a limited range of small $\Delta$ values such that $\Delta p \log(\Delta^{-1}) \leq 0.8 n$, where the 0.8 factor keeps the problem at a practically reasonable level of difficulty.  (We tried other more extreme cases, e.g., $p=500$ and $\Delta=0.1$, and found that all the methods performed poorly in terms of prediction, so apparently that setting is just too difficult.)  Here we repeat the same comparisons as in the previous section, adding comparisons across different values of $\Delta$; of course, the range of $\Delta$ for which the problem is not ultra-high-dimensional depends on $p$.  Here we found that that adaptive lasso predictions were, across the board, not competitive, so we removed it from the comparisons.  

%For the denser settings, we consider a number of different levels of sparsity, $5\%$, $10\%$, and $20\%$, with the percentages indicating the proportion of covariates to be signals, while taking the same settings otherwise, for $n$, $p$, signal size, and correlation $r$. Since we fix our sample size $n$ at 100 as we vary the number of covariates $p$, when $p=250, 500$, $20\%$ density proved to be too difficult of a task - all four of the methods being compared performed equally poorly in these settings, with most of these settings obtaining mean square prediction errors of larger than 100. And thus, in the table below, we only report results showing all three levels of sparsity for the $p=125$ setting, $5\%$ and $10\%$ density for the $p=250$ setting, and just $5\%$ density for the $p=500$ setting. In these settings, we found adaptive lasso to perform rather poorly, and thus, we excluded its results from the tables for comparison.

\begin{table}[t]
{\small 
\begin{center}  
\begin{tabular}{cccccccccccc}
    \toprule
    \multirow{2}{*}{$p$} & \multirow{2}{*}{$\Delta$} & %\multirow{2}{*}{Method} 
     &  \multicolumn{3}{c}{$A=2$}& \multicolumn{3}{c}{$A=4$}& \multicolumn{3}{c}{$A=8$}\\
     & & & $r=0.2$ & $0.5$ & $0.8$ & $r=0.2$ & $0.5$ & $0.8$ & $r=0.2$ & $0.5$ & $0.8$\\
     \midrule
    \multirow{3}{*}{125} &   \multirow{3}{*}{0.05} & EB & 1.05 & 1.16 & 1.07 & 1.00 & 1.02 & 0.90 & 1.14 & 0.94 & 1.09 \\  
     & & HS & 1.11 & 1.25 & 1.16 & 1.09 & 1.11 & 0.99 & 1.24 & 1.02 & 1.14\\
     & & Lasso &  1.41 & 1.62 & 1.46 & 1.44 & 1.59 & 1.55 & 1.77 & 1.71 & 1.75 \\
 %    & & Alasso & 1.17& 1.32 & 1.41 & 2.38 & 3.14 & 5.42 & 6.50 & 8.83 & 18.46 \\
     \hline
      \multirow{3}{*}{125} &  \multirow{3}{*}{0.1} & EB & 1.06 & 0.86 & 1.13 & 1.03 & 1.27 & 1.16 & 1.16 & 1.16 & 1.23 \\  
     & & HS & 1.17 & 1.03 & 1.32 & 1.26 & 1.56 & 1.37 & 1.35 & 1.36 & 1.44\\
     & & Lasso & 1.69 & 1.54 & 1.72 & 2.00 & 1.97 & 1.98 & 2.04 & 2.16 & 2.38 \\
 %    & & Alasso & 1.58 & 2.25 & 3.57 & 4.75 & 7.89 & 12.70 & 11.39 & 25.93 & 73.20 \\
     \hline
      \multirow{3}{*}{125} &  \multirow{3}{*}{0.2} & EB & 1.43 & 1.20 & 1.27 & 1.37 & 1.47 & 1.39 & 1.58 & 1.54 & 1.33 \\  
     & & HS & 2.53 & 2.04 & 2.04 & 2.38 & 2.08 & 1.98 & 2.28 & 2.15 & 1.96 \\
     & & Lasso & 3.34 & 2.36 & 2.23 & 3.05 & 2.59 & 2.31 & 3.20 & 2.89 & 2.52 \\
 %    & & Alasso & 2.78 & 5.85 & 15.04 & 11.31 & 28.37 & 47.30 & 50.11 & 124.06 & 332.54\\
     \hline
     \multirow{3}{*}{250} &  \multirow{3}{*}{0.05} & EB & 0.99 & 1.01 & 1.37 & 1.22 & 1.14 & 1.13 & 1.25 & 1.03 & 1.02 \\
     & &  HS & 1.11 & 1.29 & 1.60 & 1.54 & 1.38 & 1.27 & 1.56 & 1.19 & 1.29 \\
     & &  Lasso & 1.34 & 1.51 & 1.63 & 1.91 & 1.39 & 1.51 & 1.84 & 1.52 & 1.46 \\
 %    &  & Alasso & 1.41 & 2.55 & 4.46 & 3.73 & 7.64 & 15.00 & 14.17 & 26.85 & 70.88 \\
     \hline
    \multirow{3}{*}{250} &  \multirow{3}{*}{0.1} & EB & 1.28 & 1.16 & 1.11 & 1.32 & 1.47 & 1.50 & 1.33 & 1.48 & 1.27 \\
     & &  HS & 1.90 & 1.86 & 2.24 & 2.17 & 2.37 & 2.36 & 2.04 & 2.50 & 2.11 \\
     & &  Lasso & 3.69 & 1.94 & 1.42 & 3.57 & 2.57 & 2.07 & 3.34 & 2.57 & 1.80 \\
%     &  & Alasso & 3.43 & 8.12 & 13.44 & 13.04 & 23.59 & 60.54 & 49.20 & 149.68 & 277.65  \\
    \hline
     \multirow{3}{*}{500} &  \multirow{3}{*}{0.05} & EB & 1.52 & 1.18 & 1.53 & 1.46 & 1.28 & 1.40 & 1.57 & 1.70 & 1.35 \\
     & &  HS & 2.97 & 1.69 & 3.02 & 2.38 & 1.61 & 2.20 & 6.44 & 2.12 & 2.29 \\
     & &  Lasso & 9.71 & 2.65 & 1.70 & 20.69 & 2.61 & 1.75 & 41.23 & 2.92 & 1.89  \\
%     &  & Alasso & 8.22 & 6.98 & 21.95 & 29.23 & 25.64 & 59.30 & 114.03 & 102.88 & 302.51 \\
    \bottomrule
\end{tabular}
\end{center}
}
\caption{Comparison of MSPEs for the three different methods across various less-sparse settings---of dimension $p \in \{125, 250, 500\}$, density $\Delta$, signal size $A \in \{2,4,8\}$, and correlation $r \in \{0.2, 0.5, 0.8\}$---as described in the text. The standard errors are between 0.08 and 0.35, with the exception of some in the $p=500$ case. }
\label{table:mspe2}
\end{table}

From Table~\ref{table:mspe2}, EB is outperforming the other methods in almost all of these settings in terms of MSPE. Our method consistently provides low MSPE values, typically between 1 and 1.5, with the exception of a few settings when $p=500$.  Excluding the $p=500$ settings, the standard errors lie between 0.08 and 0.35, with most below or around 0.10; the mean is 0.15, and the median is 0.13. Taking these standard errors into consideration, EB and HS perform similarly when the number of signals is relatively low, in particular, at $(p,\Delta)$ pairs $(125, 0.05)$, $(125,0.1)$, and $(250, 0.05)$.  But in the moderate density cases, namely, $(p,\Delta)$ equal $(125,0.2)$ and $(250,0.1)$, the EB predictions are significantly better.  The $p=500$ case is more challenging and we find that results for all three methods are rather unstable.  Plots of the raw prediction error magnitudes (not shown) are long-tailed, especially HS and lasso, which explains the larger MSPEs and corresponding standard errors.  While there is too much instability in the $p=500$ case to make any conclusive claims of significant differences, the EB method appears to be ``less unstable'' compared to the others, which is a positive sign.  And just like in the very-sparse settings discussed above, the run-time for EB is a small fraction of that for HS.

%$p=125$ and $\Delta=0.05$ settings, but EB has significantly lower MSPE that are statistically significant for all the $p=125$, density=$20\%$ settings, as well as for the $p=250$, density=$10\%$ settings across the board. The $p=500$ settings show a little more inconsistency in all the methods, as the problem is much more challenging, and thus, the standard errors in these configurations have bigger values, ranging from 0.10 to 19.17.  

%The takeaway is that there are no instances where HS is significantly better than EB, but a number of settings where EB performs better than HS significantly.

%Some of the 250 runs in each configuration produced relatively larger prediction errors. This is especially apparent in the $p=500$ settings. The squared prediction errors have a long-tailed distribution, with all three methods having a few large squared errors from the 250 runs in each of the configurations. The spread of this distribution is larger than the one under the sparse simulation settings. There is some instability in all the algorithms, and in fact, the large errors that show up from the 250 runs do not belong to the same dataset. 

Tables~\ref{table:covg2} and~\ref{table:len2} report the prediction coverage probabilities and mean lengths, respectively, for the 95\% prediction intervals based on EB and HS.  Most of the coverage probabilities are close to the target $95\%$, however, we do see that as $p$ and/or $\Delta$ increase, a few of the coverage probabilities moved slightly further away from $95\%$.  For example, when $p=500$, $r=0.5$, and $A=8$, EB has its lowest probability of $88\%$. A few more of EB's coverage probabilities are lower than HS's, although there are also a few instances where EB's coverage is better than HS's.  Most of these differences are not statistically significant, however. The lengths of the EB prediction intervals are consistently shorter than those of HS, and in all settings where $p\geq250$, the difference between lengths of EB and HS prediction intervals are significant.

\begin{table}[t]
{\small 
\begin{center}  
\begin{tabular}{cccccccccccc}
    \toprule
    \multirow{2}{*}{$p$} & \multirow{2}{*}{$\Delta$} &
&  \multicolumn{3}{c}{$A=2$}& \multicolumn{3}{c}{$A=4$}& \multicolumn{3}{c}{$A=8$}\\
     & & & $r=0.2$ & $0.5$ & $0.8$ & $r=0.2$ & $0.5$ & $0.8$ & $r=0.2$ & $0.5$ & $0.8$\\
     \midrule
    \multirow{2}{*}{125} &   \multirow{2}{*}{0.05} & EB & 0.95 & 0.94 & 0.94 & 0.96 & 0.94 & 0.97 & 0.94 & 0.96 & 0.96 \\  
     & & HS & 0.96 & 0.94 & 0.94 & 0.95 & 0.94 & 0.97 & 0.94 & 0.96 & 0.94 \\
     \hline
      \multirow{2}{*}{125} &  \multirow{2}{*}{0.1} & EB & 0.94 & 0.97 & 0.95 & 0.96 & 0.93 & 0.95 & 0.94 & 0.95 & 0.94 \\  
     & & HS & 0.95 & 0.97 & 0.94 & 0.95 & 0.93 & 0.95 & 0.94 & 0.95 & 0.94 \\
     \hline
      \multirow{2}{*}{125} &  \multirow{2}{*}{0.2} & EB & 0.92 & 0.92 & 0.93 & 0.94 & 0.91 & 0.94 & 0.92 & 0.90 & 0.94 \\  
     & & HS & 0.91 & 0.96 & 0.94 & 0.92 & 0.94 & 0.97 & 0.92 & 0.93 & 0.95 \\
     \hline
     \multirow{2}{*}{250} &  \multirow{2}{*}{0.05} & EB & 0.95 & 0.96 & 0.92 & 0.95 & 0.94 & 0.95 & 0.91 & 0.96 & 0.94  \\
     & &  HS & 0.96 & 0.94 & 0.93 & 0.92 & 0.95 & 0.94 & 0.91 & 0.95 & 0.94 \\
     \hline
    \multirow{2}{*}{250} &  \multirow{2}{*}{0.1} & EB & 0.93 & 0.96 & 0.94 & 0.92 & 0.91 & 0.91 & 0.92 & 0.93 & 0.95 \\
     & &  HS & 0.95 & 0.96 & 0.93 & 0.94 & 0.94 & 0.91 & 0.94 & 0.91 & 0.94\\
    \hline
     \multirow{2}{*}{500} &  \multirow{2}{*}{0.05} & EB & 0.92 & 0.94 & 0.95 & 0.95 & 0.94 & 0.93 & 0.91 & 0.88 & 0.95 \\
     & &  HS & 0.96 & 0.96 & 0.94 & 0.98 & 0.96 & 0.94 & 0.94 & 0.93 & 0.90 \\
     \bottomrule
\end{tabular}
\end{center}
}
\caption{Comparison of coverage probability for the two different 95\% prediction intervals across various less-sparse settings---of dimension $p \in \{125, 250, 500\}$, density $\Delta$, signal size $A \in \{2,4,8\}$, and correlation $r \in \{0.2, 0.5, 0.8\}$---as described in the text. All standard errors are between the values of 0.008 and 0.021, with an average of 0.015. }
\label{table:covg2}
\end{table}

\begin{table}[t]
{\small 
\begin{center}  
\begin{tabular}{cccccccccccc}
    \toprule
    \multirow{2}{*}{$p$}  & \multirow{2}{*}{$\Delta$} & 
&  \multicolumn{3}{c}{$A=2$}& \multicolumn{3}{c}{$A=4$}& \multicolumn{3}{c}{$A=8$}\\
     & &  & $r=0.2$ & $0.5$ & $0.8$ & $r=0.2$ & $0.5$ & $0.8$ & $r=0.2$ & $0.5$ & $0.8$\\
     \midrule
    \multirow{3}{*}{125} &  \multirow{3}{*}{0.05} & EB & 4.15 & 4.12 & 4.17 & 4.12 & 4.09 & 4.16 & 4.13 & 4.16 & 4.17\\  
     & & HS & 4.23 & 4.17 & 4.22 & 4.19 & 4.15 & 4.21 & 4.19 & 4.22 & 4.20\\
     & & Oracle & 4.11 & 4.07 & 4.10 & 4.08 & 4.05 & 4.09 & 4.09 & 4.12 & 4.10\\
     \hline
      \multirow{3}{*}{125} &  \multirow{3}{*}{0.1} & EB & 4.14 & 4.14 & 4.20 & 4.13 & 4.15 & 4.20 & 4.15 & 4.16 & 4.17\\
     & & HS & 4.50 & 4.53 & 4.52 & 4.50 & 4.51 & 4.49 & 4.53 & 4.50 & 4.49\\
     & & Oracle & 4.22 & 4.23 & 4.21 & 4.22 & 4.23 & 4.20 & 4.24 & 4.24 & 4.19\\
     \hline
      \multirow{3}{*}{125} &  \multirow{3}{*}{0.2} & EB & 4.20 & 4.23 & 4.43 & 4.17 & 4.24 & 4.50 & 4.19 & 4.19 & 4.38 \\  
     & & HS & 5.59 & 5.59 & 5.67 & 5.59 & 5.63 & 5.73 & 5.63 & 5.53 & 5.54 \\
     & & Oracle & 4.59 & 4.58 & 4.58 & 4.56 & 4.60 & 4.65 & 4.58 & 4.54 & 4.56 \\
     \hline
     \multirow{3}{*}{250} &  \multirow{3}{*}{0.05} & EB & 4.11 & 4.16 & 4.25 & 4.12 & 4.11 & 4.22 & 4.15 & 4.14 & 4.21 \\
     & &  HS & 4.42 & 4.56 & 4.55 & 4.49 & 4.47 & 4.49 & 4.50 & 4.48 & 4.49\\
     & &  Oracle & 4.19 & 4.24 & 4.25 & 4.21 & 4.18 & 4.22 & 4.24 & 4.22 & 4.21 \\
     \hline
    \multirow{3}{*}{250} &  \multirow{3}{*}{0.1} & EB & 4.20 & 4.23 & 4.41 & 4.20 & 4.22 & 4.42 & 4.17 & 4.23 & 4.40 \\
     & &  HS & 5.65 & 5.63 & 6.04 & 5.58 & 5.59 & 5.55 & 5.55 & 5.48 & 5.61 \\
     & &  Oracle & 4.59 & 4.58 & 4.56 & 4.60 & 4.56 & 4.57 & 4.57 & 4.58 & 4.58\\
    \hline
     \multirow{3}{*}{500} &  \multirow{3}{*}{0.05} & EB & 4.26 & 4.20 & 4.54 & 4.51 & 4.20 & 4.39 & 4.31 & 4.19 & 4.44 \\
     & &  HS & 6.53 & 5.53 & 6.68 & 5.76 & 5.47 & 5.54 & 7.09 & 5.44 & 5.50\\
     & &  Oracle & 4.63 & 4.56 & 4.58 & 4.62 & 4.56 & 4.56 & 4.56 & 4.55 & 4.61 \\
    \bottomrule
\end{tabular}
\end{center}
}
\caption{Comparison of mean length for the three different 95\% prediction intervals across various less-sparse settings---of dimension $p \in \{125, 250, 500\}$, density $\Delta$, signal size $A \in \{2,4,8\}$, and correlation $r \in \{0.2, 0.5, 0.8\}$---as described in the text. All standard errors are between the values of 0.01 and 0.63, with an average of 0.04.}
\label{table:len2}
\end{table}

\subsection{Real data application}
\label{SS:real}

Following the example used in \citet{bhadra.hspred}, we use the same real-world data set to examine how our method performs. This pharmacogenomics data set is publicly available in the NCI-60 database, and can be accessed via the R package {\tt mixOmics} \citep{mixomics}, dataset {\tt multidrug}. The expression level of 12 different human ABC transporter genes are predicted using compound concentration levels. To keep our analysis on par with that in  \citet{bhadra.hspred}, we only predict with the 853 compounds that have no missing values. The data set includes 60 samples, which we randomly split into a training and testing set of 75\% and 25\%, respectively. Thus, in this regression scenario, $n=45$ and $p=853$. Each random training and testing split is performed 20 times, and we calculate the average out-of-sample MSPE for these 20 trials, shown in Table~\ref{table:real}. 

\begin{table}[t]
{\small
\begin{center}
  \begin{tabular}{lllllll}
    \toprule
    Response  & EB  & HS & Lasso & Alasso \\
    \midrule
    A1 &  0.93 & 0.93 & 0.97 & 1.00 \\  
    A2 & 0.93 & 0.94 & 1.09 & 0.99 \\
    A3 &  0.93 & 0.93 & 0.97 & 1.07 \\
    A4 &  0.93 & 0.93 & 0.88 & 1.00 \\
    A5 &  0.93 & 0.96 & 0.92 & 0.94 \\
    A6 &  0.93 & 0.93 & 0.98 & 0.93 \\
    A7 &  0.93 & 0.93 & 1.07 & 0.92 \\
    A8 &  0.93 & 0.94 & 1.06 & 1.01 \\
    A9 &  0.93 & 0.92 & 0.82 & 0.92 \\
    A10 & 0.93 & 0.93 & 0.99 & 0.93 \\
    A12 & 0.93 & 0.95 & 1.04 & 1.04 \\
    B1 & 0.73 & 0.54 & 0.61 & 0.42 \\
    \bottomrule
  \end{tabular}
\end{center}
}
\caption{Mean square prediction error for the four methods averaged over 20 random training/testing splits of the data as described in Section~\ref{SS:real}.  The rows correspond to different response variables being predicted. The standard errors range from 0.00 to 0.12, with an average of 0.03. }
\label{table:real}
\end{table}

For the 12 different transporter genes, our empirical Bayes method obtained marginally better out-of-sample MSPE in three of the genes (A2, A8, and A12) than those from the other methods implemented, while being very comparable with the other methods on the other 9 genes as response variables. For gene B1, all four methods have significantly smaller MSPE values than for the other 11 genes, with EB having the largest.  A closer look at the B1 case reveals that the estimated regression coefficients based on, e.g., lasso, are all very small, even the non-zero values.  So small, in fact, that EB tends to select no variables; consequently, its predictions on the testing set are based simply on the mean response from the training set, which is apparently not so effective.  Therefore, in cases where the signals are quite small, there is perhaps some advantage to using a continuous shrinkage-style prior like the horseshoe compared to a discrete selection-style prior like that proposed here.

%The take-away message, again, is that the empirical Bayes method is as good or better than horseshoe or lasso-based methods in terms of prediction quality, provides accurate predictive uncertainty quantification, and with lower computational cost than the horseshoe.  

\section{Discussion}
\label{S:discuss}

In this paper, we apply a recently proposed empirical Bayes method to the context of prediction in sparse high-dimensional linear regression settings.  The key idea is to let the data inform the prior center so that the tail of the prior distribution have little influence on the posterior concentration properties.  This allows for faster computation---since conjugate Gaussian priors can be used, leading to closed-form marginalization and dimension reduction---without sacrificing on posterior concentration rates.  In the context of prediction, being able to formulate the Bayesian model using conjugate priors means that the predictive distribution can be written (almost) in closed-form; it allows for some analytical integration, yielding a relatively easy to compute posterior predictive distribution for the purpose of constructing prediction intervals, etc.  We also extended the theoretical results presented in \citet{martin.mess.walker.eb} to obtain posterior concentration rates relevant to the prediction problem, for both in- and out-of-sample cases, and established a Bernstein--von Mises theorem the sheds light on the empirical Bayes posterior's potential for valid uncertainty quantification in prediction.  All these desirable features are confirmed by the results in real- and simulated-data examples.  

While the computations here using the basic MCMC on the $S$-space are relatively fast, it is worth asking if speed can be further improved if some margin of approximation is allowed.  One option is to use a {\em variational approximation} \citep[e.g.,][]{blei.etal.vb.2017} to $\Pi^n$.  Such an approach, using a point mass--Gaussian mixture mean-field approximation, was presented recently in \citet{ray.szabo.vb} but for the case where the posterior being approximated is based on independent (and data-free) Laplace priors for $\beta_S$, given $S$, as described in \citet{castillo.schmidt.vaart.reg}.  Given that the empirical prior formulation considered here is itself based on point mass--Gaussian type mixtures, the corresponding posterior is in some sense ``closer'' to this mean-field approximation.  So, in addition to the substantial decrease in computation time, some improved performance is to be expected.  Work on this is underway \citep{vebreg} and the theoretical and empirical results obtained so far confirm these expectations.  

An interesting question is if this empirical Bayes methodology can be extended to handle sparse, high-dimensional generalized linear models, such as logistic regression.  In the Gaussian setting considered here, the notion of prior centering is quite natural and relatively simple to arrange, but the idea itself is not specific to Gaussian models.  Work is underway to carry out this non-trivial extension, and we expect that similarly good theoretical and numerical results, like those obtained here for prediction, can be established in that more general context too.

\section*{Acknowledgments}

The authors thank the Action Editor and two anonymous reviewers for their helpful feedback on an earlier version of this manuscript.  This work is partially supported by the National Science Foundation, grants DMS--1737929 and DMS--1737933.

\appendix

\section{Technical details}

\subsection{Summary of results from \citet{martin.mess.walker.eb}}
\label{S:bernoulli}

For the reader's convenience, here we summarize four results from \citet{martin.mess.walker.eb}, in the context of estimation and variable selection, relevant to our present investigation about prediction.  These are not simply restatements of the results in that paper, however, here we have refined the conditions and also strengthened some of the conclusions.  

The first result concerns the concentration rate properties of the posterior distribution, $\Pi^n$, based on the empirical prior described in Section~\ref{S:review}, in the known-$\sigma$ case, with respect to a metric that focuses on estimation of the mean response.  

\begin{prop}
\label{prop:mean}
Let $\Pi^n$ be the empirical Bayes posterior defined in \eqref{eq:update}, and let $s^\star$ be a sequence such that \eqref{eq:standing} holds.  Then there exists positive constants $G$ and $M$ such that 
\[ \sup_{\beta^\star: |S_{\beta^\star}|=s^\star} \E_{\beta^\star} \Pi^n(\{\beta \in \RR^p: \|X(\beta-\beta^\star)\| > M\eps_n\}) \lesssim e^{-G n \eps_n^2} \to 0, \]
where $\eps_n^2 = n^{-1} s^\star \log (p/s^\star) \to 0$.  
\end{prop}

A small difference between Proposition~\ref{prop:mean} and the statement of Theorem~1 in \citet{martin.mess.walker.eb} is the exponential upper bound.  Their proof actually establishes the exponential bound, but they did not include this detail in their statement.  The exponential bound adds value, however, in applications like Corollary~\ref{thm:predictive.rate} in Section~\ref{SSS:rate} above.  

The above result focuses on a metric that is relevant to estimating the mean response which, of course, is relevant for our prediction context here.  But the metric's dependence on the design matrix $X$ makes ``in-sample'' prediction most natural.  For prediction at generic predictor variable settings, rates can be obtained if we focus on a different metric, namely, $\|\beta-\beta^\star\|$, which is directly related to estimation of the regression coefficients.  

Concentration rate results in this stronger metric require additional assumptions about $X$.  In particular, the sparse singular value defined in Equation~\eqref{eq:kappa} above needs to be positive for some configurations a small factor larger than the true $S^\star$.  As this is only slightly stronger than what is required for a sparse $\beta^\star$ to be identifiable \citep{ariascastro.lounici.2014}, it is not too much to ask in order to achieve accurate estimation.  The following summarizes Theorem~3 in \citet{martin.mess.walker.eb}.  Just like with Proposition~\ref{prop:mean}, here we present the result with an exponential upper bound, which is part of the proof of their original theorem, but not its statement.  

\begin{prop}
\label{prop:beta}
Let $\Pi^n$ be the empirical Bayes posterior defined in \eqref{eq:update} and $s^\star$ a sequence such that \eqref{eq:standing} holds and $\lkappa(C s^\star; X) > 0$ for $C > 2$.  Then there exists positive constants $G$ and $M$ such that 
\[ \sup_{\beta^\star: |S_{\beta^\star}|=s^\star} \E_{\beta^\star} \Pi^n\Bigl(\Bigl\{\beta \in \RR^p: \|\beta-\beta^\star\| > \frac{M \eps_n}{\lkappa(C s^\star; X)} \Bigr\}\Bigr) \lesssim e^{-G n \eps_n^2} \to 0, \]
where $\eps_n^2 = n^{-1} s^\star \log (p/s^\star) \to 0$.  
\end{prop}

Next is a basic but essential result pertaining to the posterior's ability to identify the important variables. The following proposition says that the posterior tends not to overfit, i.e., does not include unnecessary variables.  In other words, while the ambient dimension of the posterior is very high, its effective dimension is not too much larger than that of an oracle who knows the correct configuration.  One more time, the exponential upper bound here does not appear in the statement of Theorem~2 in \citet{martin.mess.walker.eb}, but their proof does establish this.  

\begin{prop}
\label{prop:dim}
Let $\Pi^n$ be the empirical Bayes posterior defined in \eqref{eq:update}, and let $s^\star$ be a sequence such that \eqref{eq:standing} holds.  Then there exists constants $C > 1$ and $G > 0$ such that 
\[ \sup_{\beta^\star: |S_{\beta^\star}|=s^\star} \E_{\beta^\star} \Pi^n(\{\beta \in \RR^p: |S_\beta| > Cs^\star \}) \lesssim e^{-G n \eps_n^2} \to 0, \]
where $\eps_n^2 = n^{-1} s^\star \log (p/s^\star) \to 0$.  
\end{prop}

Finally, the strongest of the results relevant to structure learning is the following posterior variable selection consistency theorem.  This proposition says that the posterior probability assigned to the true configuration approaches 1, which implies that any variable selection procedure derived from the full posterior, e.g., based on marginal variable inclusion probabilities or the maximum {\em a posteriori} configuration, will asymptotically identify exactly the correct set of variables. 

\begin{prop}
\label{prop:selection}
Let $\Pi^n$ be the empirical Bayes posterior defined in \eqref{eq:update} and $s^\star$ a sequence such that \eqref{eq:standing} holds and $\lkappa(C s^\star; X) > 0$ for $C > 2$.  If $\beta^\star$ is such that $|S_{\beta^\star}|=s^\star$ and  
\[ \min_{j \in S^\star} |\beta_j^\star| \geq  \frac{M \sigma}{\lkappa(C s^\star; X)} \Bigl( \frac{2\log p}{n} \Bigr)^{1/2}, \]
for some constant $M > 1$, then  $\E_{\beta^\star}\{\Pi^n(\beta: S_\beta = S_{\beta^\star})\} \to 1$ as $n \to \infty$.  
\end{prop}

The statement here differs slightly from that in \citet{martin.mess.walker.eb}, because there are some minor mistakes in their formulation.  A proof of this version of the theorem can be found in the supplementary material at \url{https://arxiv.org/abs/1406.7718v5}.  

\subsection{Predicting a $d$-dimensional response, $d > 1$}
\label{S:dim}

%{\color{red} Discuss the extension of Theorem~1 (and its corollary) to $d$-dim prediction, with $d > 1$..........?} 

In Section~\ref{SS:theory} we mentioned that, although we focus there on the case of $d=1$ for the sake of simplicity, suitable versions of the results hold for any fixed $d \geq 1$.  Here we describe how those results can be extended to the more general case.  

Let $\tilde X$ denote a generic $d \times p$ matrix, with $d > 1$, at which predictions are desired.  For coefficient vectors $\beta$ and $\beta^\star$, write $h_{\tilde X}(\beta^\star, \beta)$ for the conditional Hellinger distance between $\nm_d(\tilde X \beta, \sigma^2 I)$ and $\nm_d(\tilde X \beta^\star, \sigma^2 I)$, and like in Section~\ref{SSS:rate}, define 
\[ h(\beta^\star, \beta) = \Bigl\{\int h_{\tilde X}^2(\beta^\star, \beta) \, Q_n^d(d\tilde X) \Bigr\}^{1/2}, \]
where $Q_n^d$ is the distribution of matrices $\tilde X$ obtained by randomly sampling (without replacement) $d$ rows from the original $X$ matrix, i.e., 
\[ Q_n^d = \binom{n}{d}^{-1} \sum_{T \subset \{1,\ldots,n\}: |T|=d} \delta_{X[T,\cdot]}, \]
$\delta$ is the Dirac point mass distribution, and $X[T,\cdot]$ denotes the submatrix of $X$ obtained by keeping only the rows corresponding to indices in $T$.  The claim in Theorem~\ref{thm:hellinger.rate} pertains to a sort of marginal Hellinger distance, of which the quantity $h(\beta^\star, \beta)$ above is a generalization.  The crux of the proof of Theorem~\ref{thm:hellinger.rate} is in bounding this Hellinger distance in terms of a corresponding Kullback--Leibler divergence which has a simple form in this Gaussian setting.  If $k_{\tilde X}$ is the analogous conditional Kullback--Leibler divergence of $\nm_d(\tilde X \beta, \sigma^2 I)$ from $\nm_d(\tilde X \beta^\star, \sigma^2 I)$, then we immediately get 
\[ h_{\tilde X}(\beta^\star, \beta) \leq 2 k_{\tilde X}(\beta^\star, \beta). \]
Since the Kullback--Leibler divergence is additive for independent joint distributions, if $\tilde X$ is a realization from $Q_n^d$, i.e., if $\tilde X = X[T,\cdot]$ for a random chosen subset $T$ of indices of size $d$, then we have 
\[ k_{\tilde X}(\beta^\star, \beta) = \sum_{i \in T} k_{x_i}(\beta^\star, \beta) = \sum_{i \in T} |x_i^\top (\beta-\beta^\star)|^2, \]
where $x_i$ denotes a row of $X$ (treated as a column vector).  A key point is that this depends on $\tilde X = X[T,\cdot]$ only through $T$.  So since there are $\binom{n-1}{d-1}$ such $T$ of size $d$ that contain each row of $X$, averaging over $T$ gives 
\[ \int h_{\tilde X}^2(\beta^\star, \beta) \, Q_n^d(d\tilde X) \leq 2 \binom{n}{d}^{-1} \sum_T \sum_{i \in T} |x_i^\top (\beta-\beta^\star)|^2 = \frac{d}{n} \sum_{i=1}^n |x_i^\top(\beta-\beta^\star)|^2. \]
The right-hand side is $dn^{-1}\|X(\beta-\beta^\star)\|^2$, and since $d$ is fixed, the posterior concentration rate in terms of $h(\beta^\star,\beta)$ is no slower than that in terms of $\|X(\beta-\beta^\star)\|$, and the latter we know; see Appendix~\ref{S:bernoulli}.  So, the conclusion of Theorem~\ref{thm:hellinger.rate} holds for any $d \geq 1$.  

Corollary~\ref{thm:predictive.rate} also holds for any $d \geq 1$, with modifications like those described above when $d > 1$.  Our reason for focusing on the $d=1$ case in the main text is that the notation for and interpretation of the $Q_n^d$-type in-sample prediction is cumbersome.

%\section{Details when a prior for $\sigma$ is assumed}
%\label{S:sigma}

\bibliographystyle{apalike}
%\bibliography{/Users/rgmarti3/Dropbox/Research/mybib}
\bibliography{mybib}

\end{document}